\makeatletter
\let\my@xfloat\@xfloat
\makeatother
\documentclass[leqno]{jms-amsmath}
\usepackage{bxpapersize}
\makeatletter
\def\@xfloat#1[#2]{
    \my@xfloat#1[#2]%
    \def\baselinestretch{1}%
    \@normalsize \normalsize
}
\makeatother
\usepackage{amsmath}
\usepackage{amssymb}
\usepackage{color}
\frompage{\pageref{firstpage}}
\topage{\pageref{finalpage}}
\articlenumber{}{}{\frompage--\topage}
\shorttitle{{\it Two-sample testing based on the overlap coefficient}}
\shortauthor{Atsushi {\sc Komaba}, Hisashi {\sc Johno} and Kazunori {\sc Nakamoto}}
\title{A novel statistical approach for two-sample testing based on the overlap coefficient}
\author{Atsushi {\sc Komaba}, Hisashi {\sc Johno} and Kazunori {\sc Nakamoto}}
\usepackage{amsthm}
\usepackage{graphicx}
\usepackage{hyperref}
\usepackage{xurl}
\usepackage[capitalize]{cleveref}
\usepackage{bbm}
\usepackage{siunitx}
\usepackage{physics}
\makeatletter
\renewcommand{\@seccntformat}[1]{%
  \ifcsname format#1\endcsname
    \csname format#1\endcsname
  \else
    \csname the#1\endcsname
  \fi
  \quad
}
\@removefromreset{equation}{section}
\makeatother

\renewcommand{\thesection}{\arabic{section}}
\renewcommand{\thesubsection}{\arabic{section}.\arabic{subsection}}
\renewcommand{\theequation}{\arabic{equation}}
\newtheorem{theorem}{Theorem}[section]
\theoremstyle{definition}
\newtheorem{definition}[theorem]{Definition}
\newtheorem{proposition}[theorem]{Proposition}
\newtheorem{lemma}[theorem]{Lemma}
\newtheorem{corollary}[theorem]{Corollary}
\newtheorem*{notation}{Notation}
\theoremstyle{remark}
\newtheorem{remark}[theorem]{Remark}
\Crefname{equation}{}{}
\newcommand\A{\mathfrak{A}}
\newcommand\Z{\mathbb{Z}}
\renewcommand\d{\mathrm{d}}
\newcommand\N{\mathbb{N}}
\newcommand\R{\mathbb{R}}
\newcommand\1{\mathbbm{1}}
\newcommand\C{\mathcal{C}}
\newcommand\V{\mathcal{V}}
\newcommand\T{\mathcal{T}}
\newcommand\J{\mathcal{J}}
\renewcommand\S{\mathcal{S}}
\newcommand\wh\widehat
\newcommand\wt\widetilde
\newcommand\bs\boldsymbol
\newcommand\Seq{\text{\sc Seq}}
\newcommand\infint{\int^\infty_{-\infty}}
\newcommand\normal{\mathrm{Normal}}
\newcommand\trapezoidal{\mathrm{Trapezoidal}}
\newcommand\triangular{\mathrm{Triangular}}
\newcommand\mixed{\mathrm{Mixed}}

\DeclareMathOperator*{\argmax}{arg\,max}
\DeclareMathOperator*{\argmin}{arg\,min}
\newcommand\figscale{.8}
\begin{document}
\label{firstpage}
\maketitle
\begin{abstract}
{\bf Abstract.}\m 
  Here we propose a new nonparametric framework for two-sample testing,
  named as the OVL-$q$ ($q = 1, 2, \ldots$).
  This can be regarded as a natural extension of the Smirnov test,
  which is equivalent to the OVL-1.
  We specifically focus on the OVL-2,
  implement its fast algorithm,
  and show its superiority over other statistical tests in some experiments.

\end{abstract}
%
\label{finalpage}
\section{Introduction}
\label{sec:intro}
The overlap coefficient (OVL) is a measure of the similarity between two probability distributions,
defined as the common area under their density functions.
Previously, we have developed a nonparametric method to estimate the OVL \cite{johno21}.

In any two-sample test for equality of (continuously differentiable) distribution functions, 
the null hypothesis is equivalent to the OVL being one.
To date, however, the OVL has not been the main subject of such hypothesis testing.

The objective of this study is to construct a new nonparametric two-sample test for distribution equality based on the OVL estimation, which will be referred to as the OVL-$q$ ($q = 1, 2, \ldots$).
Furthermore, we aim to implement algorithms for the OVL-$q$
and experimentally compare the statistical power of the OVL-1 (which is equivalent to the Smirnov test) and OVL-2, for example,
with that of other statistical tests.

In this paper, we start with preliminaries and basic results in \Cref{sec:framework}.
The algorithms for the OVL-$q$ are described in \Cref{sec:algorithm}.
Experimental results are shown in \Cref{sec:exper}, and the conclusion follows in \Cref{sec:concl}.
The proofs of \Cref{thm:rhoh,thm:rho_hat.fibonacci} are given in \Cref{sec:proof.rhoh,sec:proof.rho_hat.fibonacci}, respectively.

A system to perform the OVL-1 and OVL-2 is available at \url{https://fiveseven-lambda.github.io/ovl-test/} along with its source code.

\begin{notation}
Throughout this paper,
we denote by $\Z$, $\N$, $\N_+$, $\mathbb{Q}$, and $\R$
the sets of integers, nonnegative integers, positive integers, rational numbers, and real numbers, respectively.
If $-\infty \le a \le b \le \infty$ and if there is no confusion,
we write $[a, b] = \{x : a \le x \le b\}$,
$[a, b) = \{x : a \le x < b\}$,
$(a, b] = \{x : a < x \le b\}$,
and $(a, b) = \{x : a < x < b\}$ as (extended) real intervals.
For $q \in \N_+$, we define $\R_\le^q = \{(v_1, \ldots, v_q) \in \R^q : v_1 \le \cdots \le v_q\}$.
For a set $A$, $\# A$ denotes the cardinality of $A$.
\end{notation}

\section{Analytical framework}
\label{sec:framework}
\subsection{Estimation of the OVL}
\begin{definition}
\label{def:F}
	On a probability space $(\Omega, \A, P)$,
  let $X_1, \ldots, X_m$ be real random variables
  with a continuous distribution function $F_0$,
  $Y_1, \ldots, Y_n$ be those with $F_1$,
  and $X_1, \ldots, X_m, Y_1, \ldots, Y_n$ be mutually independent.
	The empirical distribution functions
  corresponding to $\{X_1,\ldots,X_m\}$ and $\{Y_1,\ldots,Y_n\}$ are given by
	\begin{equation}
  \label{eq:def.F}
    \begin{split}
      F_{0,m}(x) &= \frac{1}{m}\sum_{i=1}^m\1_{(-\infty,x]}(X_i)\qquad (x\in\R),\\
      F_{1,n}(x) &= \frac{1}{n}\sum_{i=1}^n\1_{(-\infty,x]}(Y_i)\qquad (x\in\R),
    \end{split}
	\end{equation}
	respectively, where $\1$ denotes the indicator function.
	Put $F_0(\infty) = F_1(\infty) = F_{0, m}(\infty) = F_{1, n}(\infty) = 1$
	and $F_0(-\infty) = F_1(-\infty) = F_{0, m}(-\infty) = F_{1, n}(-\infty) \allowbreak = 0$.
\end{definition}
\begin{definition}
\label{def:rho_qmn}
	For a real function $g$ on a set $A$ and $x, y \in A$,
	we write $g|_x^y = g(y) - g(x)$.
	For $\bs{v}=(v_1,\ldots,v_q)\in\R_\le^q$, define
	\begin{align}
		r(\bs{v})
			&= \sum_{i = 0}^q
			\min\left\{ F_0|_{v_i}^{v_{i + 1}}, F_1|_{v_i}^{v_{i + 1}}\right\},
			\label{eq:def.r} \\
		r_{m, n}(\bs{v})
			&= \sum_{i = 0}^q
			\min\left\{ F_{0, m}|_{v_i}^{v_{i + 1}}, F_{1, n}|_{v_i}^{v_{i + 1}}\right\}
			\label{eq:def.r_mn},
	\end{align}
	where $v_0=-\infty$ and $v_{q+1}=\infty$.
	Note that $0\le r(\bs{v})\le 1$ and $0\le r_{m,n}(\bs{v})\le 1$ for all $\bs{v}\in\R_\le^q$.
	We also define 
	\begin{equation}
	\label{eq:def.rho_qmn}
		\rho_{q,m,n}=\min_{\bs{v}\in\R_\le^q}\,r_{m,n}(\bs{v})\in [0, 1],
	\end{equation}
	which exists because $r_{m, n}$ takes at most finitely many values.
\end{definition}
\begin{remark}
\label{rem:rho_measurable}
	Note that $\rho_{q,m,n}$ is measurable on $\Omega$,
	because
	$r_{m,n}(\bs v)$ is obviously measurable for each $\bs v \in \R_\le^q$ and
	$\R_\le^q$ in \Cref{eq:def.rho_qmn} can be replaced by its countable subset $\R_\le^q\cap\mathbb{Q}^q$
	(since $F_{0, m}$ and $F_{1, n}$ are right continuous).
\end{remark}

\begin{definition}
\label{def:ascmp}
	Suppose $\xi$ is a random variable on $(\Omega,\A,P)$
	taking values in a separable metric space $(E, d)$;
	$\{\xi_{i}:i\in\N_+\}$ and $\{\xi_{i,j}':i,j\in\N_+\}$ are two sequences of random variables on $(\Omega,\A,P)$ into $E$.
	Then we say that $\{\xi_i\}$ and $\{\xi_{i,j}'\}$ \textit{converge almost surely} to $\xi$ if
	\begin{align*}
		&P\left(\left\{\omega\in\Omega: \lim_{i\to\infty}\xi_i(\omega)=\xi(\omega)\right\}\right) = 1,\\
		&P\left(\left\{\omega\in\Omega: \lim_{i,j\to\infty}\xi_{i,j}'(\omega)=\xi(\omega)\right\}\right) = 1,
	\end{align*}
	respectively.
\end{definition}

\begin{remark}
\label{rem:f0f1}
	If $F_0$ and $F_1$ are differentiable on $\R$ with continuous derivatives $f_0$ and $f_1$, respectively,
	then the OVL between the two distributions is given by
	\begin{equation}
	\label{eq:def.rho}
		\rho = \infint\min\,\{f_0(x),f_1(x)\}\:\d x.
	\end{equation}
	We call $x\in\R$ a \textit{coincidence point} between $f_0$ and $f_1$ if $f_0(x)=f_1(x)$;
	$x\in\R$ a \textit{crossover point} between $f_0$ and $f_1$
	if there exists a neighborhood $V$ of $x$
	such that for any $a,b\in V$,
	$(a-x)(b-x)>0$ if and only if $[f_0(a)-f_1(a)][f_0(b)-f_1(b)]>0$.
	The set of crossover points and that of coincidence points are
	denoted by $C(f_0, f_1)$ and $C'(f_0, f_1)$, respectively.
	Note that $C(f_0, f_1)\subset C'(f_0, f_1)$.
\end{remark}

\begin{theorem}
\label{thm:rhoh}
	Suppose $f_0$ and $f_1$ are as in \Cref{rem:f0f1}, 
	$\#C'(f_0,f_1)<\infty$, and $\#C(f_0,f_1)=N<\infty$.
	Then $\rho_{N,m,n}$ converges almost surely to $\rho$ as $m,n\to\infty$.
\end{theorem}

See \Cref{sec:proof.rhoh} for the proof of \Cref{thm:rhoh}.

Hereafter,
$F_0$ and $F_1$ are only assumed to be continuous,
unless otherwise noted.
	
\subsection{The OVL-$q$ test}
\label{sec:OVLtest}
	For $q\in\N_+$, we define
	the OVL-$q$ test statistic as $\rho_{q,m,n}$.
	Under the null hypothesis $H_0: F_0 = F_1$,
	the p-value of $\rho_{q,m,n}$ is given by $p_{q,m,n}(\rho_{q,m,n})$ where 
	\begin{equation}
	\label{eq:def.p_qmn}
		p_{q,m,n}(x) = P(\{\omega \in \Omega : \rho_{q,m,n}(\omega) \le x\})\qquad (x\in\R),
	\end{equation}
	and the lower limit of a $100(1-\alpha)$\% confidence interval ($0<\alpha<1$) of $\rho_{q,m,n}$ is
	\begin{equation}
	\label{eq:def.l_qmn}
		l_{q,m,n}(\alpha) = \sup\,\{x\in\R:p_{q,m,n}(x)<\alpha\}.
	\end{equation}

\subsection{The Smirnov test}
\label{sec:Smirnov_test}
(See \cite{berger14} for reference.)
The Smirnov (or the two-sample Kolmogorov-Smirnov) test statistic is defined as
\[
	D_{m,n} = \max_{x\in\R}\left|F_{0,m}(x)-F_{1,n}(x)\right|.
\]

\begin{proposition}
\label{prop:rho_D}
	(See \cite[Section 3.2]{Hand} for reference.)
	The relation $\rho_{1,m,n} = 1-D_{m,n}$ holds.
\end{proposition}
\begin{proof}
	We have
	\begin{align*}
		\rho_{1, m, n} &= \min_{v \in \R}\, r_{m, n}(v) \\
		&= \min_{v \in \R} \left(
			\min\left\{F_{0, m}|_{-\infty}^v, F_{1, n}|_{-\infty}^v\right\}
			+ \min\left\{F_{0, m}|_v^\infty, F_{1, n}|_v^\infty\right\}
		\right) \\
		&= \min_{v \in \R} \left(
			\min\left\{F_{0, m}(v), F_{1, n}(v)\right\}
			+ \min\left\{1 - F_{0, m}(v), 1 - F_{1, n}(v)\right\}
		\right) \\
		&= \min_{v \in \R} \left(
			\min\left\{F_{0, m}(v), F_{1, n}(v)\right\}
			+ 1 - \max\left\{F_{0, m}(v), F_{1, n}(v)\right\}
		\right) \\
		&= \min_{v \in \R} (1 - |F_{0, m}(v) - F_{1, n}(v)|) \\
		&= 1 - \max_{v \in \R} |F_{0, m}(v) - F_{1, n}(v)| \\
		&= 1 - D_{m, n}
	\end{align*}
	by definition.
\end{proof}
Let
\begin{equation*}
	\wt{p}_{m,n}(x) = P(\left\{\omega\in\Omega:D_{m,n}(\omega)\ge x\right\})\qquad (x\in\R).
\end{equation*}
The p-value of $D_{m,n}$ under $H_0: F_0 = F_1$ is given by $\wt{p}_{m,n}(D_{m,n})$.
Since $D_{m,n}=1-\rho_{1,m,n}$ by \Cref{prop:rho_D}, we have
\begin{equation*}
	\wt{p}_{m,n}(x) = P(\left\{\omega\in\Omega:\rho_{1,m,n}(\omega)\le 1-x\right\})
	= p_{1,m,n}(1-x)\qquad (x\in\R).
\end{equation*}
Hence $\wt{p}_{m,n}(D_{m,n})$ is equivalent to the p-value of $\rho_{1,m,n}$ under $H_0$
because $$\wt{p}_{m,n}(D_{m,n})=p_{1,m,n}(1-D_{m,n})=p_{1,m,n}(\rho_{1,m,n}).$$ 
Therefore, the OVL-$1$ is equivalent to the Smirnov test.

\section{Algorithms for the OVL-$q$}
\label{sec:algorithm}

\subsection{Basic principles}

\begin{definition}
\label{def:rho_q}
	For $k\in\N_+$,
	let $\Gamma_k = \{0, 1\}^k$
	and define $N_1(\bs\gamma) = \sum_{i = 1}^k \gamma_i$
	and $N_0(\bs\gamma) = k - N_1(\bs\gamma)$
	for $\bs\gamma = (\gamma_1, \ldots, \gamma_k) \in \Gamma_k$.
	Let $\Gamma_0 = \{e\}$ where $e$ is the empty sequence,
	and define $N_0(e) = N_1(e) = 0$.
	Define $\bs\gamma_{i:j} = (\gamma_{i+1},\ldots,\gamma_j)$
	for $\bs\gamma = (\gamma_1, \ldots, \gamma_k) \in \Gamma_k$ ($k \ge 1$)
	and $i, j \in \{0, \ldots, k\}$ ($i < j$),
	and $\bs\gamma_{i:i} = e$
	for $\bs\gamma \in \Gamma_k$ ($k \ge 0$) and $i \in \{0, \ldots, k\}$.
	Let $\Gamma_{k, l} = \{\bs\gamma \in \Gamma_{k + l} : N_0(\bs\gamma) = k, N_1(\bs\gamma) = l\}$
	for $k, l\in\N$.
	For $\bs\gamma \in \Gamma_{m,n}$ and $q\in\N_+$, define
	\begin{equation}
	\label{eq:def.rho_hat}
		\widehat\rho_q(\bs\gamma) = \min_{0\le j_1\le\cdots\le j_q\le m+n} \widehat r_{\bs\gamma}(j_1,\ldots,j_q),
	\end{equation}
	where
	\begin{gather}
		\widehat r_{\bs\gamma}(j_1,\ldots,j_q)
			= \sum_{i = 0}^q \min\left\{
				\widehat F_{0, \bs\gamma} |_{j_i}^{j_{i + 1}},
				\widehat F_{1, \bs\gamma} |_{j_i}^{j_{i + 1}}
			\right\},
			\label{eq:def.r_hat} \\
		\widehat F_{0, \bs\gamma}(i) = \frac{N_0(\bs\gamma_{0:i})}{m},
			\quad
			\widehat F_{1, \bs\gamma}(i) = \frac{N_1(\bs\gamma_{0:i})}{n},
			\label{eq:def.F_hat}
	\end{gather}
	$j_0=0$, and $j_{q+1}=m+n$.
	Note that $0 \le \widehat r_{\bs\gamma}(j_1, \ldots, j_q) \le 1$,
	and hence
	\begin{equation}
	\label{eq:rho_hat.01}
		0 \le \widehat\rho_q(\bs\gamma) \le 1.
	\end{equation}
\end{definition}
Let $\widehat\Omega$ be the set of all $\omega\in\Omega$
such that $X_1(\omega),\ldots,X_m(\omega),Y_1(\omega),\ldots,\allowbreak Y_n(\omega)$ are all distinct.	
Since $F_0$ and $F_1$ are continuous, we can see that 
\begin{equation}
\label{eq:POmega1}
	P\bigl(\widehat\Omega\bigr)=1.
\end{equation}
Hence we can put $\{Z_1,\ldots,Z_{m+n}\}=\{X_1,\ldots,X_m,Y_1,\ldots,Y_n\}$ with $Z_1<\cdots<Z_{m+n}$ almost surely.
We also put $Z_0=Z_1-1$.
Now define $\widehat{\bs\gamma} = (\widehat\gamma_1, \ldots, \widehat\gamma_{m + n}) \in \Gamma_{m, n}$ on $\widehat\Omega$ by
\[
	\widehat\gamma_j = \begin{cases}
		0 &\mbox{if}\quad Z_j \in \{X_1, \ldots, X_m\}, \\
		1 &\mbox{if}\quad Z_j \in \{Y_1, \ldots, Y_n\}.
	\end{cases}
\]
\begin{remark}
\label{rem:F.F_hat}
	By \Cref{eq:def.F,eq:def.F_hat}, we have
	$ \widehat F_{0, \widehat{\bs\gamma}}(i) = F_{0, m}(Z_i) $
	and $\widehat F_{1, \widehat{\bs\gamma}}(i) = F_{1, n}(Z_i) $
	for all $i \in \{0, \ldots, m + n\}$.
\end{remark}
\begin{remark}
\label{rem:H0_prob}
	Under the null hypothesis $H_0: F_0 = F_1$, we have $\widehat{\bs\gamma}(\widehat\Omega)=\Gamma_{m,n}$ and 
	\begin{equation*}
		P(\{\omega \in \widehat\Omega : \widehat{\bs\gamma}(\omega) = \bs\gamma\})
		= (\#\Gamma_{m,n})^{-1}
		= \binom{m + n}{m}^{-1}
	\end{equation*}
	for all $\bs\gamma \in \Gamma_{m, n}$.
\end{remark}
\begin{proposition}
\label{prop:rho_hat.rho_qmn}
	For $q\in\N_+$, $\widehat\rho_q(\widehat{\bs\gamma}) = \rho_{q,m,n} \in [0,1]$.
\end{proposition}
\begin{proof}
	By \Cref{eq:def.rho_qmn,eq:def.rho_hat}, we have
	\begin{align*}
		\widehat\rho_q(\widehat{\bs\gamma})
		&= \min_{0\le j_1\le\cdots\le j_q\le m+n} \widehat r_{\widehat{\bs\gamma}} (j_1, \ldots, j_q) \\
		&= \min_{0\le j_1\le\cdots\le j_q\le m+n} r_{m, n} (Z_{j_1}, \ldots, Z_{j_q}) \\
		&= \min_{(v_1,\ldots,v_q)\in\R_\le^q} r_{m, n} (v_1, \ldots, v_q) \\
		&= \rho_{q,m,n}\in [0,1],
	\end{align*}
	noting that 
	\begin{align*}
		\widehat r_{\widehat{\bs\gamma}}(j_1, \ldots, j_q)
		&= \sum_{i = 0}^q \min\Bigl\{
			\widehat F_{0, \widehat{\bs\gamma}} \big|_{j_i}^{j_{i + 1}},
			\widehat F_{1, \widehat{\bs\gamma}} \big|_{j_i}^{j_{i + 1}}
		\Bigr\} \\
		&= \sum_{i = 0}^q \min\Bigl\{
			F_{0, m} \big|_{Z_{j_i}}^{Z_{j_{i + 1}}},
			F_{1, n} \big|_{Z_{j_i}}^{Z_{j_{i + 1}}}
		\Bigr\} \\
		&= r_{m, n}(Z_{j_1}, \ldots, Z_{j_q})
	\end{align*}
	by \Cref{eq:def.r_mn,eq:def.r_hat,rem:F.F_hat},
	where $j_0=0$ and $j_{q+1}=m+n$.
\end{proof}
\begin{theorem}
\label{thm:p_qmn.prob}
	Under the null hypothesis $H_0: F_0 = F_1$, we have
	\begin{equation*}
		p_{q,m,n}(x) = \frac{\#\left\{\bs\gamma \in \Gamma_{m,n}:\widehat\rho_q(\bs\gamma)\le x\right\}}{\#\Gamma_{m,n}}\qquad (x\in\R)
	\end{equation*}
	for $q\in\N_+$.
\end{theorem}
\begin{proof}
	This is obvious from \Cref{eq:def.p_qmn,eq:POmega1,rem:H0_prob,prop:rho_hat.rho_qmn}.
\end{proof}
	
With this theorem, we can naively perform the OVL-$q$ (see \Cref{sec:OVLtest}).
Let us call this algorithm the \textit{naive OVL-$q$}.	
If $q=2$ and $m=n$, a faster algorithm can be applied, as described in the next subsection.
An optimized algorithm for the OVL-$1$ (equivalent to the Smirnov test; see \Cref{sec:Smirnov_test}) has been previously proposed by \cite{Gunar}.

\subsection{A faster algorithm to calculate $p_{2,n,n}$}
\label{sec:FOVL}

Throughout this subsection, we assume that $m=n$ and $H_0: F_0 = F_1$ hold.
\begin{proposition}
\label{prop:rho_hat.int}
	For any $\bs\gamma \in \Gamma_{n, n}$ and $q\in\N_+$,
	there exists $k\in\{0,\ldots,n\}$ such that
	$\widehat\rho_q(\bs\gamma) = k/n$.
\end{proposition}
\begin{proof}
	It follows from \Cref{eq:def.rho_hat,eq:def.r_hat,eq:rho_hat.01} that
	\[
		\widehat\rho_q(\bs\gamma)
		= \min_{0 \le j_1 \le \cdots \le j_q \le m + n}
		\sum_{i = 0}^q
		\min\left\{
			\widehat F_{0, \bs\gamma}|_{j_i}^{j_{i + 1}},
			\widehat F_{1, \bs\gamma}|_{j_i}^{j_{i + 1}}
		\right\}
		\in [0, 1]
	\]
	where $j_0 = 0$ and $j_{q + 1} = m + n$.
	Noting that
	\[
		\min\left\{
			\widehat F_{0, \bs\gamma}|_{j_i}^{j_{i + 1}},
			\widehat F_{1, \bs\gamma}|_{j_i}^{j_{i + 1}}
		\right\} \in \left\{0,\frac{1}{n},\frac{2}{n},\ldots\right\}
	\]
	by \Cref{eq:def.F_hat},
	we obtain the claim.
\end{proof}

\begin{remark}
\label{rem:rho_hat.int}
	We can see from \Cref{prop:rho_hat.int}
	that the distribution function $p_{q,n,n}$ in \Cref{thm:p_qmn.prob} is
	uniquely determined by the values $p_{q,n,n}(k/n)$ for $k=0,\ldots,n$.
\end{remark}

\begin{definition}
\label{def:fibonacci}
	Define a sequence $\{Q_i(x)\}$ of polynomials in $x$ inductively by
	\begin{align*}
		Q_1(x) &= Q_0(x) = 1, \\
		Q_{i+2}(x) &= Q_{i+1}(x) - x Q_{i}(x) \qquad (i \in\N).
	\end{align*}
	We denote by $Q_i'(x)$ the derivative of $Q_i(x)$.
	Note that $Q_0(x),Q_1(x),\ldots$ can be regarded as formal power series.
	For a formal power series $Q(x)$, we denote by $[x^k]Q(x)$ the coefficient of $x^k$ in $Q(x)$,
	and by $1/Q(x)$ the multiplicative inverse of $Q(x)$ (if it exists).
\end{definition}

We can find $\{Q_i(x)\}$ in \cite{A115139}
as a variation of the Fibonacci polynomials.
For each $i \in \N$, we can easily see that
$[x^0]Q_i(x) = 1$, and hence $1/Q_i(x)$ exists.
			
\begin{theorem}
\label{thm:rho_hat.fibonacci}
	For $k = 0,\ldots,n$, we have
	\begin{equation}
		\label{eq:rho_hat.fibonacci}
		\#\biggl\{ \bs\gamma \in \Gamma_{n, n} : \widehat\rho_2(\bs\gamma) \ge 1 - \frac{k}{n} \biggr\}
		= [x^n]\biggl(\frac{Q_{k + 1}'(x)}{Q_{k}(x)} - \frac{Q_{k + 2}'(x)}{Q_{k + 1}(x)}\biggr).
	\end{equation}
\end{theorem}
See \Cref{sec:proof.rho_hat.fibonacci} for the proof of \Cref{thm:rho_hat.fibonacci}.

\begin{remark}
\label{rem:p2nn_fast}
	For $k = 0, \ldots, n$,
	\Cref{thm:p_qmn.prob,prop:rho_hat.int} imply
	\begin{align*}
		p_{2, n, n}\left(\frac{k}{n}\right)
		&= 1 - \frac{\#\left\{\bs\gamma \in \Gamma_{n,n}:\widehat\rho_2(\bs\gamma)> \frac{k}{n}\right\}}{\#\Gamma_{n,n}}\\
		&= 1 - \frac{\#\left\{\bs\gamma \in \Gamma_{n,n}:\widehat\rho_2(\bs\gamma)\ge \frac{k+1}{n}\right\}}{\#\Gamma_{n,n}},
	\end{align*}
	where
	\begin{equation*}
		\#\left\{\bs\gamma \in \Gamma_{n,n}:\widehat\rho_2(\bs\gamma)\ge \frac{k+1}{n}\right\}
		= [x^n] \left(\frac{Q_{n - k}'(x)}{Q_{n - k - 1}(x)} - \frac{Q_{n - k + 1}'(x)}{Q_{n - k}(x)}\right)
	\end{equation*}
	if $k\le n-1$, by \Cref{thm:rho_hat.fibonacci}.	
	It is obvious that $p_{2,n,n}(n/n)=1$.	
\end{remark}

\Cref{rem:rho_hat.int,rem:p2nn_fast} imply that we can calculate $p_{2,n,n}$ with the use of $\{Q_i(x)\}$.
Let us call this algorithm the \emph{fast OVL-$2$}.
In \Cref{sec:fast-vs-naive}, we will numerically compare the computation times of the naive and fast OVL-$2$.

\section{Numerical experiments}
\label{sec:exper}

\subsection{Computation times of the naive and fast OVL-2}
\label{sec:fast-vs-naive}
We performed the following benchmark test on a personal computer with
min \SI{2200}{\mega\hertz} -- max \SI{5083}{\mega\hertz} CPU
(AMD Ryzen 9 5950X 16-Core Processor),
\SI{62.8}{\gibi\byte} RAM,
and Linux 5.16.14 (Arch Linux).
For each $n \in \{10, 12, 14, 16\}$,
we compared the mean computation times of the naive and fast OVL-2
(averaged over 10 computations for the naive; 100000 computations for the fast)
to calculate $p_{2, n, n}(1 / 2)$.
We further measured the mean computation time of the fast OVL-2
(averaged over 10 computations)
to calculate $p_{2, n, n}(1 / 2)$ with $n \in \{500, 1000, 5000, 10000\}$.
The source code used here was written in Rust (2021 edition, rustc 1.58.1),
and is published at \url{https://github.com/fiveseven-lambda/fast-OVL-benchmark/}.

\begin{table}[t]
	\centering
	\caption{}
	\label{table:benchmark}
	\begin{tabular}{r|rr}
		\hline
		& \multicolumn{2}{c}{\:\:Mean computation time [\si{\ms}]\:\:}\\
		$n$ & naive OVL-2 & fast OVL-2 \\
		\hline
		10 & 9 & 0.026 \\
		12 & 135 & 0.026 \\
		14 & 2153 & 0.028 \\
		16 & 34361 & 0.023 \\
		500 & -- & 12 \\
		1000 & -- & 49 \\
		5000 & -- & 1865 \\
		10000 & -- & 8027 \\
		\hline
	\end{tabular}
\end{table}

\Cref{table:benchmark} shows the result of the benchmark test.
As can be seen, the fast OVL-2 was much faster than the naive OVL-2 (e.g., more than one million times faster to compute $p_{2, 16, 16}(1/2))$.
The calculation of $p_{2, n, n}(1/2)$ with $n \in \{500, 1000, 5000, 10000\}$ was computationally difficult for the naive OVL-2 but easy for the fast OVL-2 (e.g., the fast OVL-2 could compute $p_{2, 10000, 10000}(1 / 2)$ in around eight seconds).

\subsection{The statistical power of the OVL-2 test}
In this experiment, we focused on the case $m=n$ and simulated $X_1,\ldots,\allowbreak X_n,Y_1,\ldots,Y_n$ in \Cref{def:F},
with $f_0$ and $f_1$ in \Cref{rem:f0f1} being specific functions (described in the next paragraph).
The random samples were subjected to the OVL-1, OVL-2, and other statistical tests (i.e., the Welch $t$ \cite{welch47}, two-tailed $F$ \cite[Section 6.12]{snedecor89}, Mann-Whitney $U$ \cite{mann47}, and two-sample Cram\'er-von Mises test \cite{anderson62}) to verify the null hypothesis $H_0: F_0 = F_1$ with 95\% confidence interval.
This trial (from the generation of $2n$ random samples) was repeated $20000$ times independently for each $n\in\{2^2,2^3,\ldots,2^{12}\}$, and the statistical power (or equivalently the rejection ratio) of each test was calculated.
The source code used here was written in Rust (2021 edition, rustc 1.58.1) and Python (version 3.9), and is published at \url{https://github.com/fiveseven-lambda/OVL-q-test-comparison}.

As probability density functions, we used
\begin{equation*}
\normal_{\mu,\sigma}\,(x) = \frac{1}{\sqrt{2\pi}\sigma}\exp\left(-\frac{(x-\mu)^2}{2\sigma^2} \right)\qquad (x\in\R)
\end{equation*}
with $\mu\in\R$ and $\sigma>0$ for normal distributions; 
\begin{equation*}
\trapezoidal\,(x) = \begin{cases}
\ (x+2)/2 & \mbox{if}\quad -2\le x\le -\sqrt{2},\\
\ (2-\sqrt{2})/2 & \mbox{if}\quad -\sqrt{2}< x\le \sqrt{2}, \\
\ (-x+2)/2 & \mbox{if}\quad \sqrt{2}< x\le 2, \\
\ 0 & \mbox{if}\quad x<-2\mbox{ or }2<x
\end{cases}
\end{equation*}
for a trapezoidal distribution;
\begin{equation*}
\triangular\,(x) = \begin{cases}
\ (x+\sqrt{6})/6 & \mbox{if}\quad -\sqrt{6}\le x\le 0,\\
\ (-x+\sqrt{6})/6 & \mbox{if}\quad 0< x\le \sqrt{6}, \\
\ 0 & \mbox{if}\quad x<-\sqrt{6}\mbox{ or }\sqrt{6}<x
\end{cases}
\end{equation*}
for a triangular distribution;
\begin{equation*}
\mixed\,(x) = \frac{1}{2}\left(\normal_{-0.8,0.6} + \normal_{0.8,0.6}\right) (x)\qquad (x\in\R)
\end{equation*}
for a mixed normal distribution.
As a control function, we fixed $f_0=\normal_{0,1}$.
\begin{figure}[htbp]
	\centering
	\includegraphics[width = \figscale\textwidth]{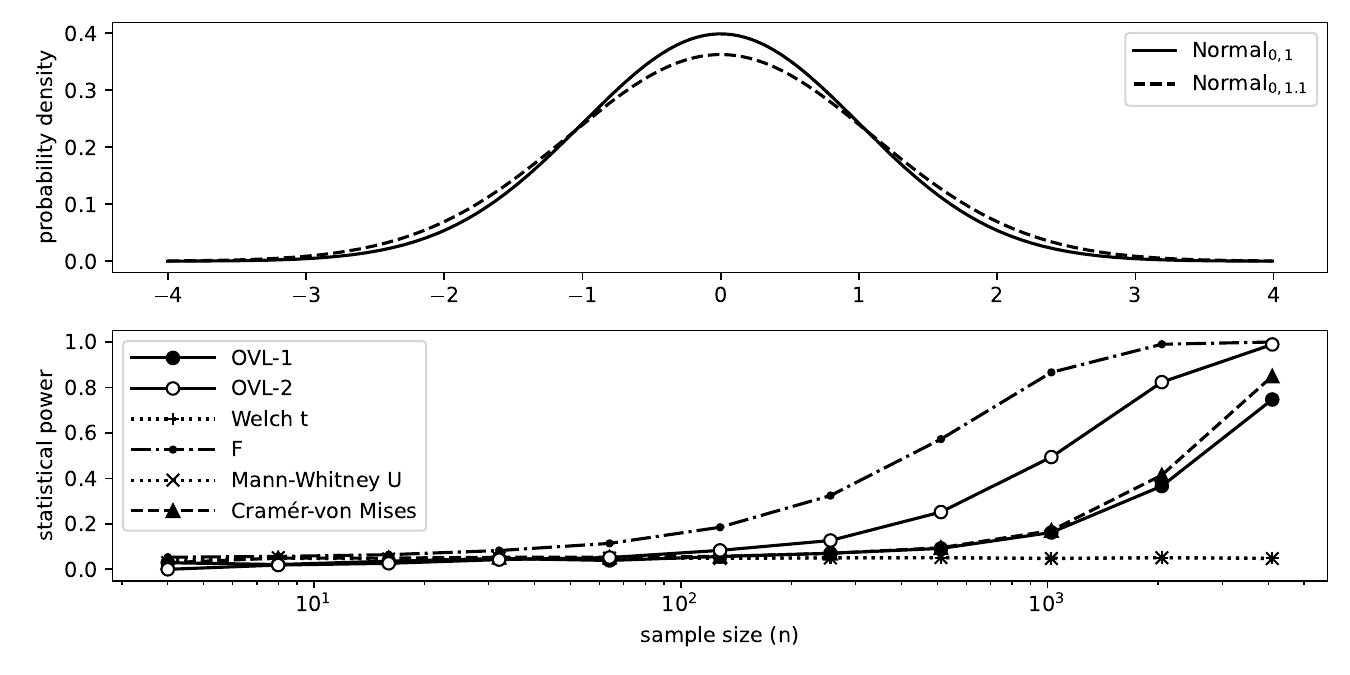}
	\caption{The random variables $X_1,\ldots,X_n,Y_1,\ldots,Y_n$ were realized with $f_0=\normal_{0,1}$ and $f_1=\normal_{0,1.1}$, and subjected to the statistical tests (the OVL-1, OVL-2, Welch $t$, $F$, Mann-Whitney $U$, and Cram\'er-von Mises test) to verify the null hypothesis $H_0: F_0 = F_1$ with 95\% confidence interval. This trial was repeated 20000 times independently for each $n\in\{2^2,2^3,\ldots,2^{12}\}$, and the statistical power of each test was evaluated. Note that $\normal_{0,1}$ has mean  $0$ and variance $1$, while $\normal_{0,1.1}$ has mean  $0$ and variance $1.21$.}
	\label{fig:1}
\end{figure}
\begin{figure}[htbp]
	\centering
	\includegraphics[width = \figscale\textwidth]{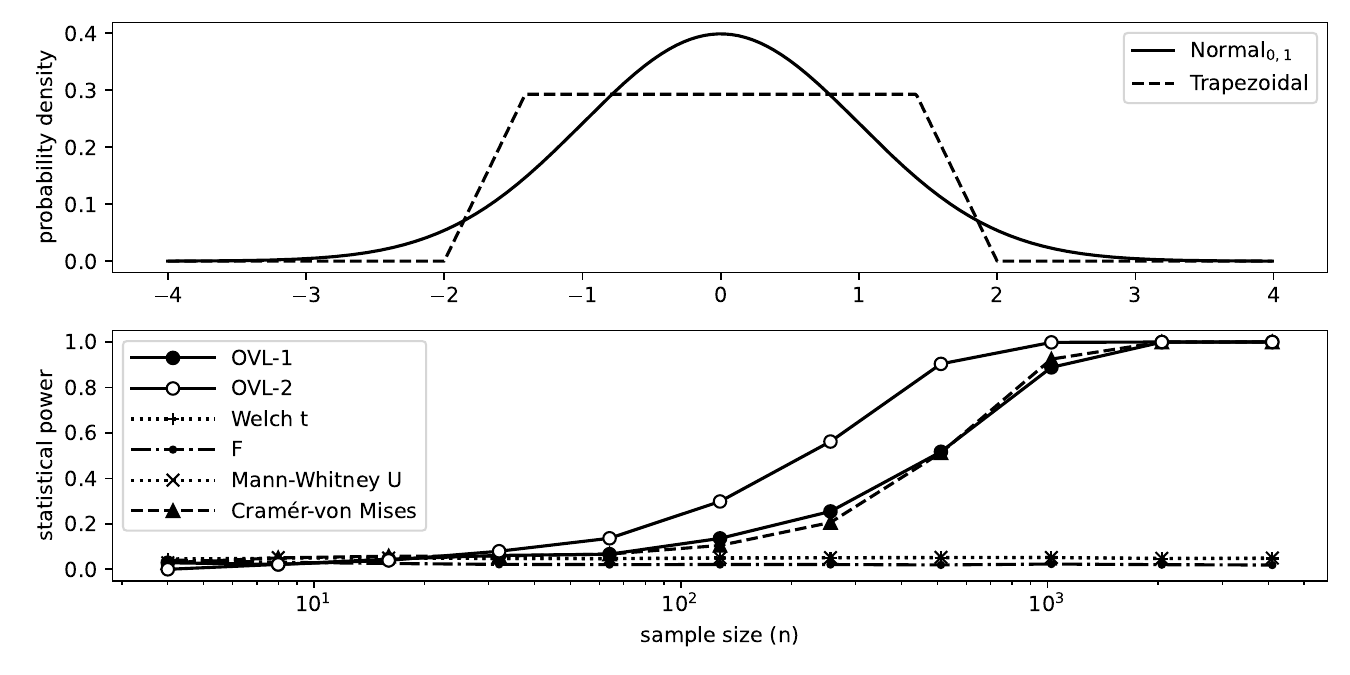}
	\caption{The random variables $X_1,\ldots,X_n,Y_1,\ldots,Y_n$ were realized with $f_0=\normal_{0,1}$ and $f_1=\trapezoidal$, and subjected to the statistical tests (the OVL-1, OVL-2, Welch $t$, $F$, Mann-Whitney $U$, and Cram\'er-von Mises test) to verify the null hypothesis $H_0: F_0 = F_1$ with 95\% confidence interval. This trial was repeated 20000 times independently for each $n\in\{2^2,2^3,\ldots,2^{12}\}$, and the statistical power of each test was evaluated. Note that $\normal_{0,1}$ and $\trapezoidal$ have mean $0$ and variance $1$.}
	\label{fig:2}
\end{figure}
\begin{figure}[htbp]
	\centering
	\includegraphics[width = \figscale\textwidth]{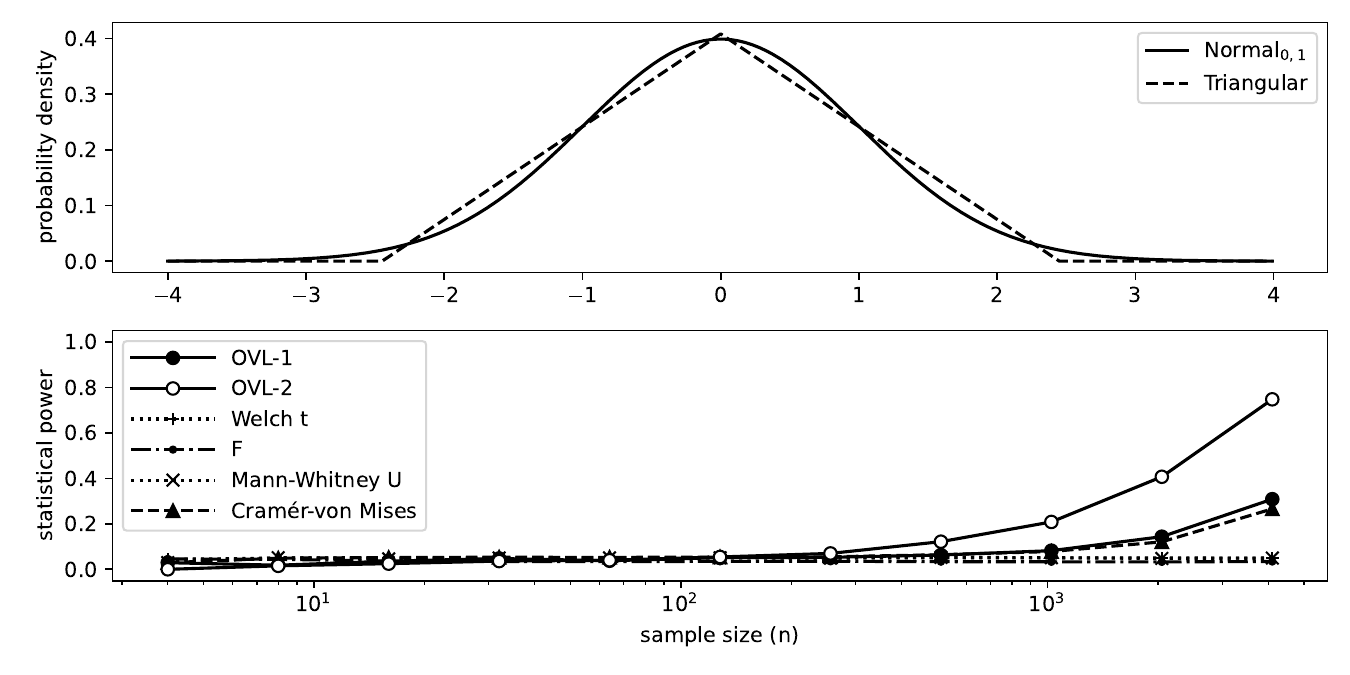}
	\caption{The random variables $X_1,\ldots,X_n,Y_1,\ldots,Y_n$ were realized with $f_0=\normal_{0,1}$ and $f_1=\triangular$, and subjected to the statistical tests (the OVL-1, OVL-2, Welch $t$, $F$, Mann-Whitney $U$, and Cram\'er-von Mises test) to verify the null hypothesis $H_0: F_0 = F_1$ with 95\% confidence interval. This trial was repeated 20000 times independently for each $n\in\{2^2,2^3,\ldots,2^{12}\}$, and the statistical power of each test was evaluated. Note that $\normal_{0,1}$ and $\triangular$ have mean $0$ and variance $1$.}
	\label{fig:3}
\end{figure}
\begin{figure}[htbp]
	\centering
	\includegraphics[width = \figscale\textwidth]{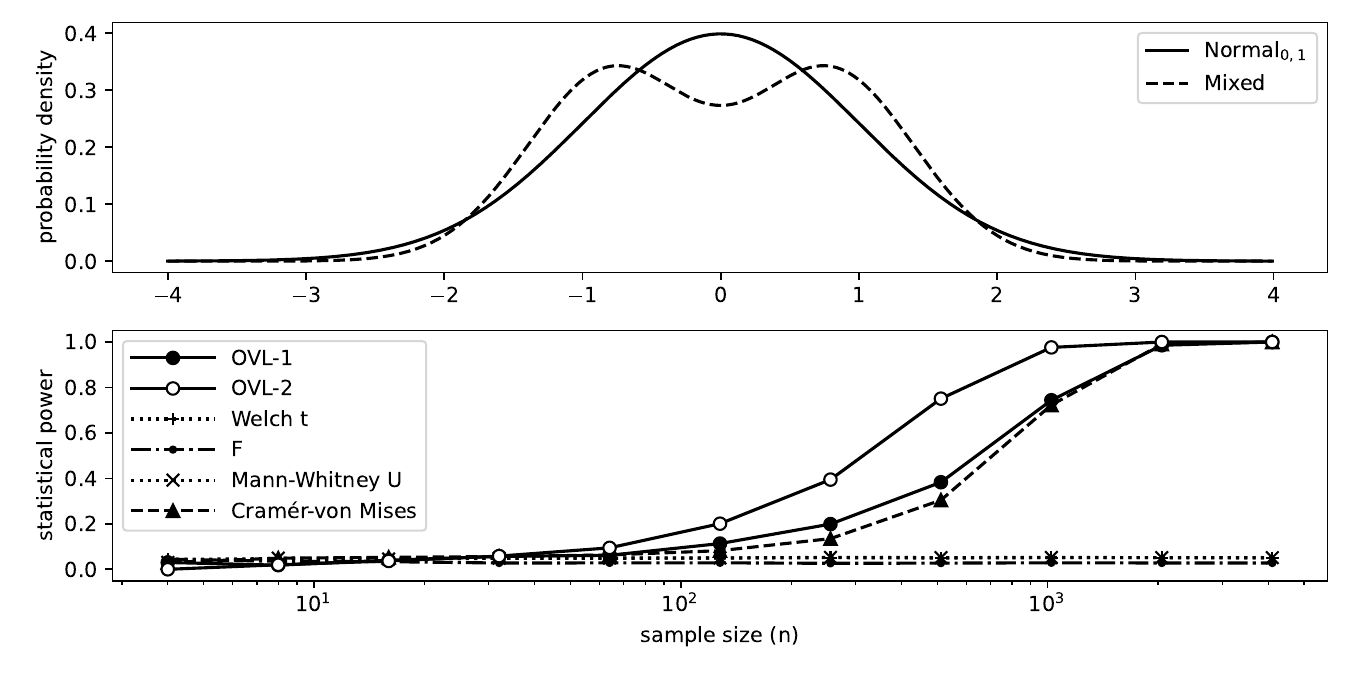}
	\caption{The random variables $X_1,\ldots,X_n,Y_1,\ldots,Y_n$ were realized with $f_0=\normal_{0,1}$ and $f_1=\mixed$, and subjected to the statistical tests (the OVL-1, OVL-2, Welch $t$, $F$, Mann-Whitney $U$, and Cram\'er-von Mises test) to verify the null hypothesis $H_0: F_0 = F_1$ with 95\% confidence interval. This trial was repeated 20000 times independently for each $n\in\{2^2,2^3,\ldots,2^{12}\}$, and the statistical power of each test was evaluated. Note that $\normal_{0,1}$ and $\mixed$ have mean $0$ and variance $1$.}
	\label{fig:4}
\end{figure}
\begin{figure}[htbp]
	\centering
	\includegraphics[width = \figscale\textwidth]{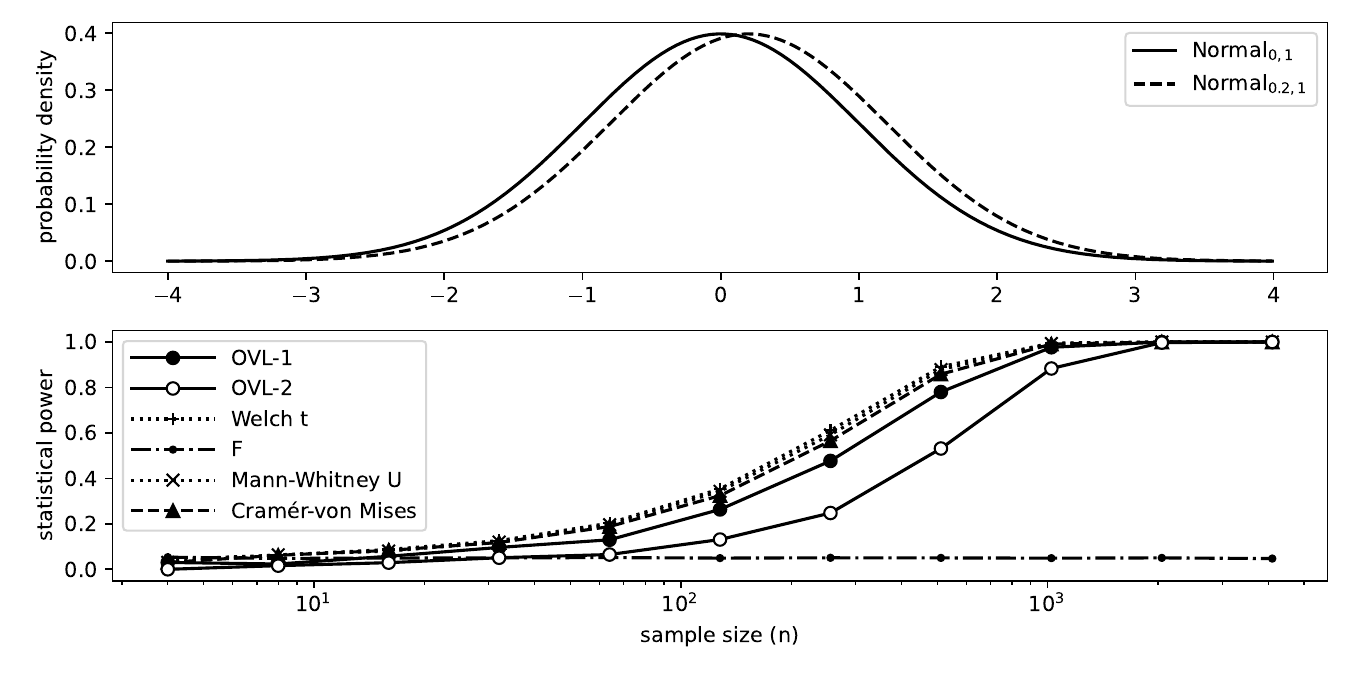}
	\caption{The random variables $X_1,\ldots,X_n,Y_1,\ldots,Y_n$ were realized with $f_0=\normal_{0,1}$ and $f_1=\normal_{0.2,1}$, and subjected to the statistical tests (the OVL-1, OVL-2, Welch $t$, $F$, Mann-Whitney $U$, and Cram\'er-von Mises test) to verify the null hypothesis $H_0: F_0 = F_1$ with 95\% confidence interval. This trial was repeated 20000 times independently for each $n\in\{2^2,2^3,\ldots,2^{12}\}$, and the statistical power of each test was evaluated. Note that $\normal_{0,1}$ has mean  $0$ and variance $1$, while $\normal_{0.2,1}$ has mean  $0.2$ and variance $1$.}
	\label{fig:5}
\end{figure}

\Cref{fig:1,fig:2,fig:3,fig:4,fig:5} show the experimental results:
\begin{itemize}
\item In the case $f_1=\normal_{0,1.1}$ where $f_0$ and $f_1$ were the densities of two normal distributions with identical means and different variances, the power of the $F$ test was the highest, followed by the OVL-2, Cram\'er-von Mises, OVL-1, and then Welch $t$ or Mann-Whitney $U$ test (\Cref{fig:1}).
\item In the case $f_1\in\{\trapezoidal,\triangular,\mixed\}$ where $f_0$ and $f_1$ were the densities of two different distributions with identical means and variances, the power of the OVL-2 test was the highest, followed by the OVL-1 or Cram\'er-von Mises, Welch $t$ or Mann-Whitney $U$, and then $F$ test (\Cref{fig:2,fig:3,fig:4}).
\item In the case $f_1=\normal_{0.2,1}$ where $f_0$ and $f_1$ were the densities of two normal distributions with different means and identical variances, the power of the Welch $t$ test was the highest, followed by the Mann-Whitney $U$, Cram\'er-von Mises, OVL-1, OVL-2, and then $F$ test (\Cref{fig:5}).
\end{itemize}

\section{Conclusion}\label{sec:concl}
Based on the OVL estimation, we have devised a novel statistical framework for two-sample testing: the OVL-$q$ ($q\in\N_+$), which can be regarded as a natural extension of the Smirnov test (since the OVL-1 is equivalent to the Smirnov test).
We have explained and implemented the algorithms for the OVL-$q$ (in particular, the fast OVL-2 algorithm).
Furthermore, we have demonstrated the superiority of the OVL-2 over conventional statistical tests in some experiments, although the reason for this is not clear at present.

One limitation is that we are currently unable to rapidly perform the OVL-2 if $m\ne n$ or the OVL-$q$ if $q\ge 3$.
To overcome this, we should explore the possibility of expanding fast and exact algorithms for the OVL-$q$, or should investigate the asymptotic distribution of $\rho_{q,m,n}$ (as $m,n\to\infty$) to approximate the OVL-$q$ in future works.
The treatment of ties (which may occur in $\Omega\setminus\wh{\Omega}$ if $F_0$ or $F_1$ is practically discontinuous) is also an important topic of research.
In addition, it is meaningful to further evaluate the statistical power of the OVL-$q$ both in simulations and in real observations.


\section{Proof for \Cref{thm:rhoh}}
\label{sec:proof.rhoh}
\begin{definition}\label{def:ascmp2}
In the setting of \Cref{def:ascmp}, we say that $\{\xi_i\}$ \textit{converges completely} to $\xi$ if
\begin{equation*}
\sum_{i=1}^\infty P\left(\left\{\omega\in\Omega: d\left(\xi_i(\omega),\xi(\omega)\right)>\epsilon\right\}\right) < \infty
\end{equation*}
for any $\epsilon>0$.
\end{definition}

\begin{remark}\label{rem:ascmp}
{\normalfont (See \cite{hsu47} for reference.)}
It is well known that $\{\xi_i\}$ converges almost surely to $\xi$ if and only if
\begin{equation*}
\lim_{l\to\infty}P\left(\bigcup_{i=l}^\infty\left\{\omega\in\Omega:d\left(\xi_i(\omega),\xi(\omega)\right)>\epsilon\right\}\right)=0
\end{equation*}
for any $\epsilon>0$.
Note that if $\{\xi_i\}$ converges completely to $\xi$, then $\{\xi_i\}$ converges almost surely to $\xi$.
\end{remark}
\begin{theorem}\label{Fhconv}
{\normalfont (The Glivenko-Cantelli theorem. See \cite[Theorem A, Section 2.1.4]{serfling80} for the proof.)}
As $m\to
\infty$ and $n\to\infty$, $$\sup_{x\in\R}|F_{0,m}(x)-F_0(x)|,\qquad \sup_{x\in\R}|F_{1,n}(x)-F_1(x)|$$ converge completely to $0$, respectively.
\end{theorem}

\begin{lemma}\label{xyzw}
{\normalfont (See \cite[Lemma A.12]{johno21} for the proof.)}
If $x,y,z,w\in\R$, then 
\begin{itemize}
	\item[\rm{(a)}]\ $|\max\,\{x,y\}-\max\,\{z,w\}|\le |x-z|+|y-w|$,
	\item[\rm{(b)}]\ $|\min\,\{x,y\}-\min\,\{z,w\}|\le |x-z|+|y-w|$.
\end{itemize}
\end{lemma}

In accordance with \Cref{thm:rhoh}, let
$F_0$ and $F_1$ be differentiable on $\R$ with continuous derivatives $f_0$ and $f_1$, respectively, $\#C'(f_0,f_1)<\infty$, $C(f_0,f_1)=\{c_1,\ldots,c_N\}$ with $c_1<\cdots<c_N$, $\bs{c}=(c_1,\ldots,c_N)$, $c_0=-\infty$, and $c_{N+1}=\infty$.

\begin{remark}\label{rem:rc_rho}
It follows from \Cref{eq:def.rho,def:rho_qmn} that $r(\bs{c})=\rho$.
\end{remark}

\begin{definition}\label{def:r_rmn}
For $q\in\N_+$, define
\begin{align*}
&\V_q = \argmin_{\bs{v}\in\R_\le^q}\,r(\bs{v}),\\
&\V_{q,m,n} = \argmin_{\bs{v}\in\R_\le^q}\,r_{m,n}(\bs{v}),\\
&\C_q = \{(c_{i_1},\ldots,c_{i_q}):1\le i_1<\cdots<i_q\le N\}.
\end{align*}
\end{definition}

\begin{remark}\label{VmCm}
It follows from \Cref{eq:def.rho_qmn,thm:rho_app,hwhv2,hwhv3} that $\V_q\ne\emptyset$ and $\V_{q,m,n}\ne\emptyset$ for all $q\in\N_+$.
It is obvious that $\C_N=\{\bs{c}\}$ and $\C_q=\emptyset$ if $q>N$.
\end{remark}


\begin{lemma}\label{lem:vcv}
Suppose $q\in\N_+$, $\bs{v}=(v_1,\ldots,v_q)\in\R_\le^q$, $v_0=-\infty$, and $v_{q+1}=\infty$.
If $v_{i}<c_s<v_{i+1}$ for some $i\in\{0,\ldots,q\}$ and $s\in\{1,\ldots,N\}$, then $r(\bs{v})>\rho$.
\end{lemma}
\begin{proof}
Since $\#C'(f_0,f_1)<\infty$, there is an open interval $U\subset (v_{i}, v_{i+1})$ with $U\cap C'(f_0, f_1) = \{c_s\}$, so that $[f_0(a)-f_1(a)][f_0(b)-f_1(b)]<0$ for all $a,b\in U$ with $a<c_s<b$.
Now fix such $a$ and $b$.
Without loss of generality, we assume that $f_0(a)<f_1(a)$ and $f_0(b)>f_1(b)$.
If $ F_0|_{v_i}^{v_{i+1}}\le  F_1|_{v_i}^{v_{i+1}}$, then 
\begin{align*}
&r(\bs{v})-\rho\\
&= \sum_{j=0}^{q}\Biggl(\min\left\{\int_{v_{j}}^{v_{j+1}}f_0(x)\:\d x, \int_{v_{j}}^{v_{j+1}}f_1(x)\:\d x\right\} \\
&\omit\hfill$\displaystyle - \int_{v_{j}}^{v_{j+1}}\min\,\{f_0(x),f_1(x)\}\:\d x\Biggr) $\\
&\ge \min\left\{\int_{v_{i}}^{v_{i+1}}f_0(x)\:\d x, \int_{v_{i}}^{v_{i+1}}f_1(x)\:\d x\right\} - \int_{v_{i}}^{v_{i+1}}\min\,\{f_0(x),f_1(x)\}\:\d x\\
&= \int_{v_{i}}^{v_{i+1}}\left(f_0(x) - \min\,\{f_0(x),f_1(x)\}\right)\:\d x\\
&\ge \int_{c_s}^{b}\left(f_0(x) - f_1(x)\right)\:\d x\\
&> 0.
\end{align*}
We can similarly prove that $r(\bs{v})-\rho>0$ if $ F_0|_{v_i}^{v_{i+1}}>  F_1|_{v_i}^{v_{i+1}}$.
\end{proof}

\begin{theorem}\label{thm:rho_app}
The minimum of $r$ on $\R_\le^N$ is $\rho$, which is uniquely attained at $\bs{c}$.
\end{theorem}
\begin{proof}
If $\bs{v}=(v_1,\ldots,v_N)\in\R_\le^N$ with $\bs{v}\ne\bs{c}$, then $c_s\notin\{v_1,\dots,v_N\}$ for some $s$, hence $v_{i}<c_s<v_{i+1}$ for some $i$ as in the assumption of \Cref{lem:vcv}, so that $r(\bs{v})>\rho=r(\bs{c})$ (note \Cref{rem:rc_rho}).
\end{proof}

\begin{theorem}\label{delISc2}
If $q\in\{1,\ldots,N-1\}$ and $\bs{v}\in\R_\le^q$, then $r(\bs{v})>r(\bs{c})$.
\end{theorem}
\begin{proof}
Since $q<N$, $c_s\notin\{v_1,\dots,v_q\}$ for some $s$.
The proof is similar as that of \Cref{thm:rho_app}.
\end{proof}

\begin{theorem}\label{hwhv}
If $q\in\{1,\ldots,N-1\}$, then for any $\bs{v}=(v_1,\ldots,v_q)\in\R_\le^q$, there exists $\bs{w}=(c_{j_1},\ldots,c_{j_q})$ with $1\le j_1\le\cdots\le j_q\le N$ such that $r(\bs{w})\le r(\bs{v})$.
\end{theorem}
\begin{proof}
Let $\bs{v}=(v_1,\ldots,v_q)\in\R_\le^q$, $v_0=-\infty$, $v_{q+1}=\infty$, and $$\eta(\bs{v})=\#\{i\in\{1,\ldots,q\}: v_i\notin C(f_0, f_1)\}.$$
The statement obviously holds when $\eta(\bs{v})=0$.
Hence suppose $\eta(\bs{v})>0$.
	Then we can choose $i\in\{1,\ldots,q\}$ and $s\in\{1,\ldots,N\}$ satisfying $c_{s-1}<v_i<c_s\le v_{i+1}$ or $v_{i-1}\le c_s<v_i<c_{s+1}$.
	We will only prove the case $c_{s-1}<v_i<c_s\le v_{i+1}$, as the other is similar.
Without loss of generality, we may assume that $f_0\le f_1$ on $(c_{s-1}, c_s)$, so that $ F_0|_{c_{s-1}}^{v_i}<  F_1|_{c_{s-1}}^{v_i}$ and $ F_0|_{v_i}^{c_s}<  F_1|_{v_i}^{c_s}$, since $\#C'(f_0,f_1)<\infty$.
In the following, we consider the cases (I) $ F_0|_{v_{i-1}}^{v_{i}}\le F_1|_{v_{i-1}}^{v_{i}}$ and (II) $ F_0|_{v_{i-1}}^{v_{i}}> F_1|_{v_{i-1}}^{v_{i}}$.

(I) Suppose $ F_0|_{v_{i-1}}^{v_{i}}\le F_1|_{v_{i-1}}^{v_{i}}$.
Then
\begin{align*}
& F_0|_{v_{i-1}}^{c_s}< F_1|_{v_{i-1}}^{c_s},\\
& F_j|_{v_{i-1}}^{c_s}= F_j|_{v_{i-1}}^{v_i}+ F_j|_{v_i}^{c_s}\qquad (j=0,1),\\
& F_j|_{c_s}^{v_{i+1}}= F_j|_{v_i}^{v_{i+1}}- F_j|_{v_i}^{c_s}\qquad (j=0,1),
\end{align*}
hence
\begin{align*}
\min_j\, F_j|_{v_{i-1}}^{c_s} + \min_j\, F_j|_{c_s}^{v_{i+1}}
&= F_0|_{v_{i-1}}^{c_s} + \min_j\,(F_j|_{v_i}^{v_{i+1}}- F_j|_{v_i}^{c_s})\\
&= F_0|_{v_{i-1}}^{v_i} +  F_0|_{v_i}^{c_s} + \min_j\,( F_j|_{v_i}^{v_{i+1}}- F_j|_{v_i}^{c_s})\\
&\le F_0|_{v_{i-1}}^{v_i} +  F_0|_{v_i}^{c_s} + \min_j\, F_j|_{v_i}^{v_{i+1}}- F_0|_{v_i}^{c_s}\\
&= F_0|_{v_{i-1}}^{v_i} + \min_j\, F_j|_{v_i}^{v_{i+1}}\\
&=\min_j\, F_j|_{v_{i-1}}^{v_i} + \min_j\, F_j|_{v_i}^{v_{i+1}},
\end{align*}
and setting $\bs{v}'=(v_1,\ldots,v_{i-1},c_s,v_{i+1},\ldots,v_q)\in\R_\le^q$ results in $\eta(\bs{v}')<\eta(\bs{v})$ and $r(\bs{v}')\le r(\bs{v})$.

(II) Suppose $ F_0|_{v_{i-1}}^{v_{i}}> F_1|_{v_{i-1}}^{v_{i}}$.
Since $f_0\le f_1$ on $(c_{s-1},c_s)$, we can see that $v_{i-1}<c_{s-1}<v_i$ and $ F_0|_{v_{i-1}}^{c_{s-1}}> F_1|_{v_{i-1}}^{c_{s-1}}$.
(II-i) First consider the case $ F_0|_{v_i}^{v_{i+1}}\le F_1|_{v_i}^{v_{i+1}}$.
Then $ F_0|_{c_{s-1}}^{v_{i+1}}< F_1|_{c_{s-1}}^{v_{i+1}}$, hence
\begin{align*}
\min_j\, F_j|_{v_{i-1}}^{c_{s-1}} + \min_j\, F_j|_{c_{s-1}}^{v_{i+1}}
&= F_1|_{v_{i-1}}^{c_{s-1}} +  F_0|_{c_{s-1}}^{v_{i+1}}\\
&= F_1|_{v_{i-1}}^{c_{s-1}} +  F_0|_{c_{s-1}}^{v_i} +  F_0|_{v_i}^{v_{i+1}}\\
&< F_1|_{v_{i-1}}^{c_{s-1}} +  F_1|_{c_{s-1}}^{v_i} +  F_0|_{v_i}^{v_{i+1}}\\
&= F_1|_{v_{i-1}}^{v_i} +  F_0|_{v_i}^{v_{i+1}}\\
&=\min_j\, F_j|_{v_{i-1}}^{v_i} + \min_j\, F_j|_{v_i}^{v_{i+1}},
\end{align*}
and setting $\bs{v}'=(v_1,\ldots,v_{i-1},c_{s-1},v_{i+1},\ldots,v_q)\in\R_\le^q$ results in $\eta(\bs{v}')<\eta(\bs{v})$ and $r(\bs{v}')< r(\bs{v})$.
(II-ii) Next consider the case $ F_0|_{v_i}^{v_{i+1}}> F_1|_{v_i}^{v_{i+1}}$.
(II-ii-a) If there is $x\in (c_{s-1},v_i)$ such that $ F_0|_{x}^{v_{i+1}}\le F_1|_{x}^{v_{i+1}}$, then $ F_0|_{v_{i-1}}^{x}> F_1|_{v_{i-1}}^{x}$, hence the case (II-i) applies to $\bs{v}''=(v_1,\ldots,v_{i-1},x,v_{i+1},\ldots,v_q)\in\R_\le^q$, where $\eta(\bs{v}'')=\eta(\bs{v})$ and 
\begin{align*}
r(\bs{v}'')-r(\bs{v})
&=\min_{j}\, F_j|_{v_{i-1}}^{x} + \min_{j}\, F_j|_{x}^{v_{i+1}} - \min_{j}\, F_j|_{v_{i-1}}^{v_i} - \min_{j}\, F_j|_{v_i}^{v_{i+1}}\\
&= F_1|_{v_{i-1}}^{x} +  F_0|_{x}^{v_{i+1}} -  F_1|_{v_{i-1}}^{v_i} -  F_1|_{v_i}^{v_{i+1}}\\
&\le F_1|_{v_{i-1}}^{x} +  F_1|_{x}^{v_{i+1}} -  F_1|_{v_{i-1}}^{v_i} -  F_1|_{v_i}^{v_{i+1}}\\
&=0.
\end{align*}
(II-ii-b) If $ F_0|_{x}^{v_{i+1}}> F_1|_{x}^{v_{i+1}}$ for any $x\in (c_{s-1},v_i)$, then $ F_0|_{c_{s-1}}^{v_{i+1}}\ge F_1|_{c_{s-1}}^{v_{i+1}}$, and setting $\bs{v}'=(v_1,\ldots,v_{i-1},c_{s-1},v_{i+1},\ldots,v_q)\in\R_\le^q$ results in $\eta(\bs{v}')<\eta(\bs{v})$ and 
\begin{align*}
r(\bs{v}')-r(\bs{v})
&=\min_{j}\, F_j|_{v_{i-1}}^{c_{s-1}} + \min_{j}\, F_j|_{c_{s-1}}^{v_{i+1}} - \min_{j}\, F_j|_{v_{i-1}}^{v_i} - \min_{j}\, F_j|_{v_i}^{v_{i+1}}\\
&= F_1|_{v_{i-1}}^{c_{s-1}} +  F_1|_{c_{s-1}}^{v_{i+1}} -  F_1|_{v_{i-1}}^{v_i} -  F_1|_{v_i}^{v_{i+1}}\\
&=0.
\end{align*}
Taken together, for any $\bs{v}\in\R_\le^q$ with $\eta(\bs{v})>0$, there exists $\bs{v}'\in\R_\le^q$ such that $\eta(\bs{v}')<\eta(\bs{v})$ and $r(\bs{v}')\le r(\bs{v})$.
The statement follows by induction.
\end{proof}

\begin{corollary}\label{hwhv2}
If $q\in\{1,\ldots,N-1\}$, then there exists $\bs{c}'\in\C_q$ such that $r(\bs{c}')=\inf\,\{r(\bs{v}): \bs{v}\in\R_\le^q\}$.
Furthermore, $r(\bs{c}')>r(\bs{c})$.
\end{corollary}
\begin{proof}
Since there are only finitely many choices for $\bs{w}\in\R_\le^q$ in \Cref{hwhv}, we can choose $\bs{w}'=(c_{i_1},\ldots,c_{i_q})\in\argmin_{\bs{w}} r(\bs{w})$, where $\bs{w}$ ranges over the choices.
Then $r(\bs{w}')\le r(\bs{v})$ for all $\bs{v}\in\R_\le^q$.
Suppose $\bs{w}'\notin\C_q$ and put $A=\{c_{i_1},\ldots,c_{i_q}\}$.
Then $\#A<q$, and there exists $A'=\{c_{j_1},\ldots,c_{j_q}\}$ such that $A\subset A'$ and $1\le j_1<\cdots<j_q\le N$.
Putting $\bs{c}'=(c_{j_1},\ldots,c_{j_q})$, we have $\bs{c}'\in\C_q$ and $r(\bs{c}')\le r(\bs{w}')$ by definition.
Hence $r(\bs{c}')=r(\bs{w}')=\min\,\{r(\bs{v}): \bs{v}\in\R_\le^q\}$.
Furthermore, $r(\bs{c}')>r(\bs{c})$ by \Cref{delISc2}.
\end{proof}

\begin{theorem}\label{hwhv3}
For $q=N+1,N+2,\ldots$, the minimum of $r$ on $\R_\le^q$ is $\rho$.
\end{theorem}
\begin{proof}
Since $\{c_1,\ldots,c_N\}\subset\{v_1,\ldots,v_q\}$ implies $r(\bs{c})=r(v_1,\ldots,v_q)$, the claim follows by \Cref{rem:rc_rho,lem:vcv}.
\end{proof}

\begin{remark}
For some $q\in\{1,\ldots,N-1\}$, $\bs{v}\in\V_q$ does not necessarily imply $\bs{v}\in\C_q$.
(Note that $\V_N=\C_N=\{\bs{c}\}$ by \Cref{thm:rho_app} and $\C_{N+1}=\C_{N+2}=\cdots=\emptyset$.)
	Here we give an example for the case where $N = 3$, $q = 2$, and $\V_2 \not\subset \C_2$.
Assume that $f_0$ and $f_1$ are defined by
\begin{equation*}
	f_0(x) = \begin{cases}
		\frac{1 - \cos x}{4\pi} & (0 \le x \le 4\pi), \\
		0 & (\text{otherwise}),
	\end{cases}
	\qquad
	f_1(x) = \begin{cases}
		\frac{1 + \cos x}{4\pi} & (\pi \le x \le 5\pi), \\
		0 & (\text{otherwise}).
	\end{cases}
\end{equation*}
	Then $(3\pi/2, 11)$ is in $\V_2$ but not in $\C_2 = \{(3\pi/2, 5\pi/2), (3\pi/2, 7\pi/2), (5\pi/2, \allowbreak 7\pi/2)\}$.
	(Note that $7\pi/2 = 10.995\ldots < 11$.)
\end{remark}

The measurability of $\sup_{\bs{v}\in\R_\le^q}|r_{m,n}(\bs{v})-r(\bs{v})|$ on $\Omega$ can be proved by the same argument as in \Cref{rem:rho_measurable}.

\begin{theorem}\label{thm:suphconv}
For $q\in\N_+$, $\sup_{\bs{v}\in\R_\le^q}|r_{m,n}(\bs{v})-r(\bs{v})|$ converges almost surely to $0$ as $m,n\to\infty$.
\end{theorem}
\begin{proof}
For $\bs{v}\in\R_\le^q$, we have
\begin{align*}
\left|r_{m,n}(\bs{v})-r(\bs{v})\right| 
&\le \sum_{i=0}^{q}\left|\min\left\{F_{0,m}|_{v_i}^{ v_{i+1}},F_{1,n}|_{v_i}^{ v_{i+1}}\right\} - \min\left\{F_{0}|_{v_i}^{ v_{i+1}},F_{1}|_{v_i}^{ v_{i+1}}\right\}\right|  \\
&\le \sum_{i=0}^{q}\left(\Bigl|F_{0,m}|_{v_i}^{ v_{i+1}}-F_{0}|_{v_i}^{ v_{i+1}}\Bigr| + \Bigl|F_{1,n}|_{v_i}^{ v_{i+1}} - F_{1}|_{v_i}^{ v_{i+1}}\Bigr|\right)
\end{align*}
by definition and \Cref{xyzw}.
Since
\begin{align*}
\Bigl|F_{0,m}|_{v_i}^{ v_{i+1}}-F_{0}|_{v_i}^{ v_{i+1}}\Bigr|
&\le  \left|F_{0,m}(v_{i+1})-F_0(v_{i+1})\right| + \left|F_{0,m}(v_{i}) - F_0(v_{i})\right|\\
&\le 2\sup_{x\in\R}\left|F_{0,m}(x) - F_0(x)\right|,\\
\Bigl|F_{1,n}|_{v_i}^{ v_{i+1}}-F_{1}|_{v_i}^{ v_{i+1}}\Bigr|
&\le  \left|F_{1,n}(v_{i+1})-F_1(v_{i+1})\right| + \left|F_{1,n}(v_{i}) - F_1(v_{i})\right|\\
&\le 2\sup_{x\in\R}\left|F_{1,n}(x) - F_1(x)\right|,
\end{align*}
we obtain
\begin{align*}
& \sup_{\bs{v}\in\R_\le^q}\left|r_{m,n}(\bs{v})-r(\bs{v})\right| \\
& \le 2(q+1)\left(\sup_{x\in\R}\left|F_{0,m}(x) - F_0(x)\right| + \sup_{x\in\R}\left|F_{1,n}(x) - F_1(x)\right|\right),
\end{align*}
	whose right side converges almost surely to $0$ as $m,n\to\infty$ by \Cref{rem:ascmp,Fhconv}.
\end{proof}

\begin{definition}
Let $(A, d)$ be a metric space.
We define a discrepancy of $A_1\subset A$ from $A_2\subset A$ by
\begin{equation*}
D(A_1, A_2)=\sup_{a_1\in A_1}\left(\inf_{a_2\in A_2}d(a_1,a_2)\right).
\end{equation*}
If the metric space is $\R^q$ $(q\in\N_+)$ with the Euclidean metric, we write $D_q$ in place of $D$.
\end{definition}

\begin{lemma}\label{supnotK}
Let $(A, d)$ be a metric space.
Let $g$ and $g_{i,j}$ ($i,j\in\N_+$) be real functions on $A$ such that $\min\,\{g(t): t\in A\}$ and $\min\,\{g_{i,j}(t):t\in A\}$ exist.
Put $T=\argmin_{t\in A}\,g(t)$ and $T_{i,j}=\argmin_{t\in A}\,g_{i,j}(t)$.
Suppose $g$ is continuous on $A$, $\sup_{t\in A}|g_{i,j}(t)-g(t)|\to 0$ as $i,j\to\infty$, and there is a compact set $K\subset A$ such that
\begin{equation*}
\min\left\{g(t): t\in A\right\} < \inf\left\{g(t): t\in A\setminus K\right\}.
\end{equation*}
Then $D(T_{i,j},T)\to 0$ as $i,j\to\infty$.
\end{lemma}
\begin{proof}
If $T'=\argmax_{t\in A}\,(-g(t))$ and $T_{i,j}'=\argmax_{t\in A}\,(-g_{i,j}(t))$, then $T'=T$ and $T_{i,j}'=T_{i,j}$, hence $D(T_{i,j},T)=D(T_{i,j}',T')\to 0$ by \cite[Lemma A.15]{johno21} (replace $g$ and $g_{i}$ with $-g$ and $-g_{i,j}$, respectively).
\end{proof}

\begin{lemma}\label{Kexist}
There exists a compact set $K\subset\R_\le^N$ such that
\begin{equation*}
\min\left\{r(\bs{v}):\bs{v}\in\R_\le^N\right\}
< \inf\left\{r(\bs{v}): \bs{v}\in\R_\le^N\setminus K\right\}.
\end{equation*}
\end{lemma}
\begin{proof}
By \Cref{thm:rho_app,hwhv2}, there exist 
\begin{equation*}
M_q=\min\left\{r(\bs{v}): \bs{v}\in\R_\le^q\right\}\qquad (q=1\ldots,N)
\end{equation*}
and $M=\min\,\{M_1,\ldots,M_{N-1}\}>M_N$.
Choose $\epsilon>0$ with $\epsilon<(M-M_N)/3$.
We can take $\alpha,\beta\in\R$ with $\alpha<\beta$ such that $F_j(\alpha)<\epsilon$ and $1-F_j(\beta)<\epsilon$ ($j=0, 1$), since $F_j$ are nondecreasing functions with $\lim_{x\to -\infty}F_j(x)=0$ and $\lim_{x\to \infty}F_j(x)=1$.
Let $K=[\alpha, \beta]^N\cap\R_\le^N$ and $\bs{v}=(v_1,\ldots,v_N)\in\R_\le^N\setminus K$.
Then $v_1<\alpha$ or $v_N>\beta$ holds.

Suppose $v_1<\alpha$ and put $\bs{v}'=(v_2,\dots,v_N)$. 
Using \Cref{xyzw}, we obtain
\begin{align*}
\left|r(\bs{v})-r(\bs{v}') \right|
&= \left|\min_{j}\, F_j|_{-\infty}^{v_1} + \min_{j}\, F_j|_{v_1}^{v_2} - \min_{j}\, F_j|_{-\infty}^{v_2} \right| \\
&\le \left|\min_{j}\, F_j|_{-\infty}^{v_1}\right| + \left|\min_{j}\, F_j|_{v_1}^{v_2} - \min_{j}\, F_j|_{-\infty}^{v_2} \right| \\
	&<\epsilon + \Bigl| F_0|_{v_1}^{v_2} - F_0|_{-\infty}^{v_2} \Bigr| + \Bigl| F_1|_{v_1}^{v_2} -  F_1|_{-\infty}^{v_2} \Bigr| \\
	&=\epsilon + \Bigl| F_0|_{-\infty}^{v_1} \Bigr| + \Bigl| F_1|_{-\infty}^{v_1} \Bigr| \\
&< 3\epsilon.
\end{align*}
Hence $M\le r(\bs{v}')\le |r(\bs{v}')-r(\bs{v})| + r(\bs{v}) < 3\epsilon + r(\bs{v})$.
We can similarly prove that $M< 3\epsilon + r(\bs{v})$ for the case $v_N>\beta$.

Therefore $M < 3\epsilon + r(\bs{v})$ for all $\bs{v}\in\R_\le^N\setminus K$, so that $$M_N<M-3\epsilon\le \inf\,\{r(\bs{v}): \bs{v}\in\R_\le^N\setminus K\}$$ holds.
This is the claim.
\end{proof}

The measurability of $D_N(\V_{N,m,n}, \V_N)$ will be proved at the end of this section.

\begin{theorem}\label{DVn}
As $m,n\to\infty$, $D_N(\V_{N,m,n}, \V_N)$ converges almost surely to $0$.
\end{theorem}
\begin{proof}
In \Cref{supnotK}, let $(A, d)$ be the subspace $\R_\le^N$ of the Euclidean metric space $\R^N$, $g=r$ (which is continuous on $\R_\le^N$), and $g_{i,j}=r_{i,j}$.
Then by \Cref{VmCm,thm:suphconv,Kexist}, the assumptions in \Cref{supnotK} are satisfied almost surely, hence $D_N(\V_{N,m,n}, \V_N)$ converges almost surely to $0$ as $m,n\to\infty$.
\end{proof}

\begin{corollary}\label{v_conv_app}
As $m,n\to\infty$, $\bs{v}_{m,n}\in\V_{N,m,n}$ converges almost surely to $\bs{c}$.
\end{corollary}
\begin{proof}
Since $\V_N=\{\bs{c}\}$ by \Cref{thm:rho_app}, we have  $$D_N(\V_{N,m,n}, \V_N) = \sup_{\bs{v}\in\V_{N,m,n}}d(\bs{v},\bs{c}) \ge d(\bs{v}_{m,n},\bs{c}).$$
Hence the claim follows from \Cref{DVn}.
\end{proof}

We cannot guarantee that $\bs{v}_{m,n}$ is necessarily measurable.
We mean by ``$\bs{v}_{m,n}$ converges almost surely to $\bs{c}$'' that there exists a measurable set $A\subset\{\omega\in\Omega :\lim_{m,n\to\infty}\bs{v}_{m,n}=\bs{c}\}$ with $P(A)=1$.
In fact, we can take $A=\{\omega\in\Omega :\lim_{m,n\to\infty}D_N(\V_{N,m,n}, \V_N)=0\}$.
If $(\Omega, \A, P)$ is complete, we have $P(\{\omega \in \Omega : \lim_{m, n \to \infty} \bs{v}_{m, n} = \bs{c}\}) = 1$.

\begin{theorem}\label{thm:rhoh2}
As $m,n\to\infty$, $\rho_{N,m,n}$ converges almost surely to $\rho$.
\end{theorem}
\begin{proof}
Let $\bs{v}_{m,n}=(v_1,\ldots,v_N)\in\V_{N,m,n}$, $v_0=-\infty$, and $v_{N+1}=\infty$.
By \Cref{def:rho_qmn,xyzw,thm:rho_app}, we have 
\begin{align*}
\left|\rho_{N,m,n}-\rho\right|
&= \left|r_{m,n}(\bs{v}_{m,n})-r(\bs{c})\right|\\
&\le \sum_{i=0}^N \left(\Bigl|F_{0,m}|_{v_i}^{ v_{i+1}}-F_{0}|_{c_i}^{ c_{i+1}}\Bigr| + \Bigl|F_{1,n}|_{v_i}^{ v_{i+1}} - F_{1}|_{c_i}^{ c_{i+1}}\Bigr|\right),
\end{align*}
where
\begin{align*}
\Bigl|F_{0,m}|_{v_i}^{ v_{i+1}}-F_{0}|_{c_i}^{ c_{i+1}}\Bigr|
&= \left|F_{0,m}(v_{i+1})-F_{0,m}(v_{i}) - F_0(c_{i+1})+F_0(c_{i}) \right|\\
&\le \left|F_{0,m}(v_{i+1}) - F_{0}(v_{i+1})\right|
+ \left|F_{0}(v_{i+1}) - F_0(c_{i+1})\right|\\
&\qquad+ \left|F_{0,m}(v_{i}) - F_{0}(v_{i})\right|
+ \left|F_{0}(v_{i}) - F_0(c_{i})\right|
\end{align*}
and 
\begin{align*}
\Bigl|F_{1,n}|_{v_i}^{ v_{i+1}}-F_{1}|_{c_i}^{ c_{i+1}}\Bigr|
&= \left|F_{1,n}(v_{i+1})-F_{1,n}(v_{i}) - F_1(c_{i+1})+F_1(c_{i}) \right|\\
&\le \left|F_{1,n}(v_{i+1}) - F_{1}(v_{i+1})\right|
+ \left|F_{1}(v_{i+1}) - F_1(c_{i+1})\right|\\
&\qquad+ \left|F_{1,n}(v_{i}) - F_{1}(v_{i})\right|
+ \left|F_{1}(v_{i}) - F_1(c_{i})\right|.
\end{align*}
Now recall that we are considering the probability space $(\Omega,\A,P)$.
By \Cref{rem:ascmp,Fhconv}, there exists $A_1\in\A$ with $P(A_1)=1$ such that for each $\omega\in A_1$, $$\sup_{x\in\R}|F_{0,m}(x)-F_0(x)|\to 0,\qquad\sup_{x\in\R}|F_{1,n}(x)-F_1(x)|\to 0$$ as $m,n\to\infty$.
Since $F_0$ and $F_1$ are continuous on $\R$ and $\bs{v}_{m,n}=(v_1,\ldots,v_N)\allowbreak\to\bs{c}$ almost surely as $m,n\to\infty$ by \Cref{v_conv_app}, there exists $A_2\in\A$ with $P(A_2)=1$ such that for each $\omega\in A_2$, $$|F_k(v_i) - F_k(c_i)|\to 0\qquad (k=0,1;\ i=1,\ldots,N)$$ as $m,n\to\infty$.
Put $A=A_1\cap A_2$.
Then $P(A)=1$.
For each $\omega\in A$ and for any $\epsilon>0$, there exist integers $N_1(\omega),N_2(\omega)$ such that $m,n\ge N_1(\omega)$ implies
$$\sup_{x\in\R}|F_{0,m}(x)-F_0(x)|<\epsilon,\qquad\sup_{x\in\R}|F_{1,n}(x)-F_1(x)|<\epsilon$$
and $m,n\ge N_2(\omega)$ implies
$$|F_k(v_i) - F_k(c_i)|<\epsilon\qquad (k=0,1;\ i=1,\ldots,N),$$
hence $m,n\ge N(\omega)=\max\,\{N_1(\omega),N_2(\omega)\}$ implies
\begin{equation*}
\Bigl|F_{0,m}|_{v_i}^{ v_{i+1}}-F_{0}|_{c_i}^{ c_{i+1}}\Bigr|
+
\Bigl|F_{1,n}|_{v_i}^{ v_{i+1}}-F_{1}|_{c_i}^{ c_{i+1}}\Bigr|
<8\epsilon
\end{equation*}
for $i=0,\ldots,N$, and therefore
\begin{equation*}
\left|\rho_{N,m,n}(\omega)-\rho\right| < \sum_{i=0}^N 8\epsilon = 8(N+1)\epsilon.
\end{equation*}
Since $\epsilon$ was arbitrary, $\rho_{N,m,n}\to\rho$ almost surely as $m,n\to\infty$.
\end{proof}

Note that \Cref{thm:rhoh2} is exactly \Cref{thm:rhoh}.

Hereafter, we prove that $D_N(\V_{N,m,n}, \V_N)$ is measurable on $\Omega$ (related to \Cref{DVn}).

\begin{definition}
Let $(Z_1,\gamma_1),\ldots,(Z_{m+n},\gamma_{m+n})$ be the rearrangement of $(X_1,0),\ldots,(X_m,0),(Y_1,1),\ldots,(Y_n,1)$ with $Z_1\le\cdots\le Z_{m+n}$,
$\T$ be the set of tuples $(t_1,\ldots,t_l)$ of positive integers with $t_1+\cdots+t_l=m+n$,
	$\R_{(t_1,\ldots,t_l)}^{m+n}$ the set of real $m+n$-tuples $(v_1,\ldots,v_{m+n})$ with $v_1=\cdots=v_{t_1}<v_{t_1+1}=\cdots =v_{t_1+t_2}<\cdots <v_{t_1+\cdots+t_{l-1}+1}=\cdots=v_{m+n}$, 
and $\S_{(t_1,\ldots,t_l)}=(\{0,1\}^{t_1}/\sim)\times\cdots\times(\{0,1\}^{t_l}/\sim)$, where $\{0,1\}^{t}/\sim$ denotes the $t$-th symmetric product of $\{0,1\}$.
For $\bs{t}\in \T$ and $\bs{s}\in \S_{\bs t}$, let $\Omega_{\bs t}=\{\omega\in\Omega:(Z_1,\ldots,Z_{m+n})\in\R_{\bs t}^{m+n}\}$, $\Omega_{\bs{t},\bs{s}}=\{\omega\in\Omega_{\bs t}:(\gamma_1,\ldots,\gamma_{m+n})\text{ corresponds to }\bs{s}\}$.
Put $I_0 = (-\infty,Z_1)$ and $I_i = [Z_i, Z_{i+1})$ for $i=1,\ldots,m+n$ where $Z_{m+n+1}=\infty$.
Denote by $\J$ the set of $N$-tuples $(j_1,\ldots,j_N)$ of integers with $0\le j_1\le\cdots\le j_N\le m+n$.
For $(j_1\ldots,j_N)\in \J$, define $I_{(j_1,\ldots,j_N)} = (I_{j_1}\times\cdots\times I_{j_N})\cap\R_\le^N$.
\end{definition}

\begin{remark}\label{rem:D_Nmeas}
Since $\R_{\bs{t}}^{m+n}$ are measurable and pairwise disjoint for $\bs{t}\in \T$, so are $\Omega_{\bs t}$.
In addition, $\R_\le^{m+n}=\bigcup_{\bs{t}\in \T}\R_{\bs{t}}^{m+n}$ implies that $\Omega = \bigcup_{\bs{t}\in \T}\Omega_{\bs t}$, where $\Omega_{\bs t}$ equals the disjoint union of $\Omega_{\bs{t},\bs{s}}\in\A$ over $\bs{s}\in \S_{\bs t}$.
Besides, for any $\bs{t}\in\T$, there exists a nonempty set $\J_{\bs t}\subset \J$ such that, on the event $\Omega_{\bs t}$, $\R_\le^N$ equals the disjoint union of $I_{\bs j}$ ($\ne\emptyset$) over $\bs{j}\in \J_{\bs t}$.

	For $\bs{t}\in\T$ and $\bs{s}\in\S_{\bs t}$, consider the event $\Omega_{\bs{t},\bs{s}}$ in the case $\Omega_{\bs{t},\bs{s}} \ne \emptyset$.
Then $r_{m,n}$ is constant on $I_{\bs j}$ for any $\bs{j}\in\J_{\bs t}$, since it depends only on the rank statistics of $X_1,\ldots,X_m,Y_1,\ldots,Y_n$.
Furthermore, there exists a nonempty set $\J_{\bs{t},\bs{s}}\subset \J_{\bs t}$ such that $\V_{N,m,n}=\argmin_{\bs{v}\in\R_\le^N}\,r_{m,n}(\bs{v})$ equals the disjoint union of $I_{\bs j}$ ($\ne\emptyset$) over $\bs{j}\in \J_{\bs{t},\bs{s}}$.
\end{remark}

\begin{theorem}
$D_N(\V_{N,m,n}, \V_N)$ is measurable on $\Omega$.
\end{theorem}
\begin{proof}
Let $\bs{t}\in\T$ and $\bs{s}\in\S_{\bs t}$.
Since $\Omega$ equals the disjoint union of $\Omega_{\bs{t},\bs{s}}\in\A$ over $\bs{t}\in \T$ and $\bs{s}\in \S_{\bs t}$ by \Cref{rem:D_Nmeas}, it suffices to prove the measurability of $D_N(\V_{N,m,n}, \V_N)$ on each $\Omega_{\bs{t},\bs{s}}$.
In the rest of the proof, we restrict $D_N(\V_{N,m,n}, \V_N)$ to $\Omega_{\bs{t},\bs{s}}$, which gives $D_N(\V_{N,m,n}, \V_N)=\max_{\bs{j}\in\J_{\bs{t},\bs{s}}} D_N(I_{\bs{j}}, \allowbreak\{\bs{c}\})$.

If there exists $\bs{j}=(j_1,\ldots,j_N)\in\J_{\bs{t},\bs{s}}$ such that $j_1=0$ or $j_N=m+n$, then $I_{\bs{j}} = ((-\infty, Z_1)\times I_{j_2}\times\cdots\times I_{j_N})\cap\R_\le^N$ or $I_{\bs{j}} = (I_{j_1}\times\cdots\times I_{j_{N-1}}\times[Z_{m+n},\infty))\cap\R_\le^N$, so that $D_N(\V_{N,m,n}, \V_N)=\infty$ (the measurability is obvious).

Next, let $\bs{j}=(j_1,\ldots,j_N)\in\J_{\bs{t},\bs{s}}$ with $1\le j_1\le\cdots\le j_N\le m+n-1$.
Then the closure $\overline{I_{\bs j}}$ of $I_{\bs{j}}$ equals $([Z_{j_1},Z_{j_1+1}]\times\cdots\times[Z_{j_N},Z_{j_N+1}])\cap\R_\le^N$.
Since $I_{\bs j}\ne\emptyset$, we have $Z_{j_k}<Z_{j_k+1}$ for all $k = 1, \ldots, N$.
Put $V_{\bs j} = \{(v_1,\ldots,v_N)\in \overline{I_{\bs j}}:v_i\in\{Z_{j_i}, Z_{j_i+1}\} \text{ for } i=1,\ldots,N\}$.
We can see that $\overline{I_{\bs j}}$ is the convex hull of the finite vertex set $V_{\bs j}$, so that $\sup_{\bs{v}\in\overline{I_{\bs{j}}}} d(\bs{v},\bs{c}) = \max_{\bs{v}\in V_{\bs j}}d(\bs{v}, \bs{c})$ since $\overline{I_{\bs{j}}}\supset V_{\bs j}$ by definition and the closed ball with center $\bs{c}$ and radius $\max_{\bs{v}\in V_{\bs j}}d(\bs{v}, \bs{c})$ contains $\overline{I_{\bs{j}}}$.
Noting that $\sup_{\bs{v}\in I_{\bs j}} d(\bs{v},\bs{c}) = \sup_{\bs{v}\in\overline{I_{\bs{j}}}} d(\bs{v},\bs{c})$, we obtain $D_N(I_{\bs{j}}, \{\bs{c}\}) = \max_{\bs{v}\in V_{\bs j}}d(\bs{v}, \bs{c})$.
Since $d(\bs{v}, \bs{c})$ for each $\bs{v}\in V_{\bs j}$ is obviously measurable on $\Omega_{\bs{t},\bs{s}}$, $D_N(I_{\bs{j}}, \{\bs{c}\})$ and also $D_N(\V_{N,m,n}, \V_N)$ are measurable on $\Omega_{\bs{t},\bs{s}}$.
\end{proof}


\section{Proof for \Cref{thm:rho_hat.fibonacci}}
\label{sec:proof.rho_hat.fibonacci}
Here we assume the same setting as in \Cref{sec:FOVL}.
We denote by $R[x]$ and $R[[x]]$ the rings of polynomials and formal power series in $x$ over a ring $R$, respectively.

\begin{definition}
For $\bs\gamma \in \Gamma_{k,k}$ ($k\in\N$), define
\begin{align*}
	\delta_{\bs\gamma}(i) &= N_0(\bs\gamma_{0:i}) - N_1(\bs{\gamma}_{0:i})
	&& (i = 0, 1, \ldots, 2k), \\
	d_{\bs\gamma}(i, j) &= | \delta_{\bs\gamma}(i) | + |\delta_{\bs\gamma}(i) - \delta_{\bs\gamma}(j)| + | \delta_{\bs\gamma}(j)|
	&& (i, j = 0, 1, \ldots, 2k).
\end{align*}
Note that $\delta_{\bs\gamma}(0) = \delta_{\bs\gamma}(2k) = 0$, 
	\begin{equation}\label{eq:deltaF}
		\delta_{\bs\gamma}(i) = k\left(\wh F_{0, \bs\gamma}(i) - \wh F_{1, \bs\gamma}(i)\right) \qquad (k > 0),
	\end{equation}
	and
	\begin{equation}\label{eq:dsym}
		d_{\bs\gamma}(i, j) = d_{\bs\gamma}(j, i)
	\end{equation} 
	by definition.
\end{definition}

\begin{lemma}\label{thm:rho2=maxd}
	For all $\bs\gamma \in \Gamma_{n, n}$,
	\[ \wh\rho_2(\bs\gamma) = 1 - \frac1{2n}\max_{0 \le j_1, j_2 \le 2n} d_{\bs\gamma}(j_1, j_2). \]
\end{lemma}
\begin{proof}
	Let us put 
	\begin{align*}
		\wh s_{\bs\gamma}(j_1, j_2)
		&= \sum_{i = 0}^2 \max\left\{ \wh{F}_{0, \bs\gamma}|_{j_i}^{j_{i + 1}}, \wh{F}_{1, \bs\gamma}|_{j_i}^{j_{i + 1}} \right\}\qquad(0\le j_1\le j_2\le 2n),
	\end{align*}
	where $j_0 = 0$, $j_3 = 2n$. Then
	\begin{align*}
		& \wh s_{\bs\gamma}(j_1, j_2) + \wh r_{\bs\gamma}(j_1, j_2) \\
		&= \sum_{i = 0}^2
		\left(\max\left\{ \wh{F}_{0, \bs\gamma}|_{j_i}^{j_{i + 1}}, \wh{F}_{1, \bs\gamma}|_{j_i}^{j_{i + 1}} \right\}
		+ \min\left\{ \wh{F}_{0, \bs\gamma}|_{j_i}^{j_{i + 1}}, \wh{F}_{1, \bs\gamma}|_{j_i}^{j_{i + 1}} \right\}\right) \\
		&= \sum_{i = 0}^2 \left( \wh{F}_{0, \bs\gamma}|_{j_i}^{j_{i + 1}} + \wh{F}_{1, \bs\gamma}|_{j_i}^{j_{i + 1}} \right) \\
		&= 2.
	\end{align*}
	On the other hand, by \Cref{eq:deltaF},
	\begin{align*}
		& \wh s_{\bs\gamma}(j_1, j_2) - \wh r_{\bs\gamma}(j_1, j_2) \\
		&= \sum_{i = 0}^2
		\left(\max\left\{ \wh{F}_{0, \bs\gamma}|_{j_i}^{j_{i + 1}}, \wh{F}_{1, \bs\gamma}|_{j_i}^{j_{i + 1}} \right\}
		- \min\left\{ \wh{F}_{0, \bs\gamma}|_{j_i}^{j_{i + 1}}, \wh{F}_{1, \bs\gamma}|_{j_i}^{j_{i + 1}} \right\}\right) \\
		&= \sum_{i = 0}^2 \left| \wh{F}_{0, \bs\gamma}|_{j_i}^{j_{i + 1}} - \wh{F}_{1, \bs\gamma}|_{j_i}^{j_{i + 1}} \right| \\
		&= \frac{1}{n}\sum_{i = 0}^2 \left|\delta_{\bs\gamma}(j_{i+1}) -  \delta_{\bs\gamma}(j_i)\right|\\
		&= \frac{d_{\bs\gamma}(j_1, j_2)}n.
	\end{align*}
	Hence we have
	\begin{equation*}
		\wh r_{\bs\gamma}(j_1, j_2)
		= 1 - \frac{d_{\bs\gamma}(j_1, j_2)}{2n},
	\end{equation*}
	so that
	\begin{align*}
		\wh\rho_2(\bs\gamma)
		&= \min_{0 \le j_1 \le j_2 \le 2n} \wh r_{\bs\gamma}(j_1, j_2) \\
		&= 1 - \frac1{2n}\max_{0 \le j_1 \le j_2 \le 2n} d_{\bs\gamma}(j_1, j_2)\\
		&= 1 - \frac1{2n}\max_{0 \le j_1, j_2 \le 2n} d_{\bs\gamma}(j_1, j_2)
	\end{align*}
	by \Cref{eq:dsym}.
\end{proof}
\begin{definition}
	For $\bs\gamma \in \Gamma_{k, k}$ ($k\in\N$), define
\begin{align*}
	\overline\delta_{\bs\gamma} = \max_{0 \leq i \leq 2k} \delta_{\bs\gamma}(i), \qquad
	\underline\delta_{\bs\gamma} = \min_{0 \leq i \leq 2k} \delta_{\bs\gamma}(i).
\end{align*}
	Note that $\underline\delta_{\bs\gamma} \le 0 \le \overline\delta_{\bs\gamma}$ since $\delta_{\bs\gamma}(0) = 0$.
\end{definition}

\begin{lemma}\label{thm:maxd=diff}
	For all $\bs\gamma \in \Gamma_{n, n}$,
	\begin{align*}
		\max_{0 \leq i, j \leq 2n} d_{\bs\gamma}(i, j) = 2\bigl(\overline\delta_{\bs\gamma} - \underline\delta_{\bs\gamma}\bigr).
	\end{align*}
\end{lemma}
\begin{proof}
	Denote $\overline\delta_{\bs\gamma, i, j} = \max\,\{\delta_{\bs\gamma}(i), \delta_{\bs\gamma}(j)\} $
	and $\underline\delta_{\bs\gamma, i, j} = \min\,\{\delta_{\bs\gamma}(i), \delta_{\bs\gamma}(j)\} $.
	Note that $\overline\delta_{\bs\gamma, i, j} + \underline\delta_{\bs\gamma, i, j} = \delta_{\bs\gamma}(i) + \delta_{\bs\gamma}(j)$
	and $\overline\delta_{\bs\gamma, i, j} - \underline\delta_{\bs\gamma, i, j} = |\delta_{\bs\gamma}(i) - \delta_{\bs\gamma}(j)|$.
	
	If $\delta_{\bs\gamma}(i) > 0$ and $\delta_{\bs\gamma}(j) > 0$, then
	\begin{align*}
		d_{\bs\gamma}(i, j)
		&= \delta_{\bs\gamma}(i) + \delta_{\bs\gamma}(j)
		+ |\delta_{\bs\gamma}(i) - \delta_{\bs\gamma}(j)| \\
		&= \overline\delta_{\bs\gamma, i, j} + \underline\delta_{\bs\gamma, i, j}
		+ \overline\delta_{\bs\gamma, i, j} - \underline\delta_{\bs\gamma, i, j} \\
		&= 2\overline\delta_{\bs\gamma, i, j} \\
		&\le 2\overline\delta_{\bs\gamma} \\
		&\le 2\bigl(\overline\delta_{\bs\gamma} - \underline\delta_{\bs\gamma}\bigr).
	\end{align*}
	
	If $\delta_{\bs\gamma}(i) < 0$ and $\delta_{\bs\gamma}(j) < 0$, then
	\begin{align*}
		d_{\bs\gamma}(i, j)
		&= - \delta_{\bs\gamma}(i) - \delta_{\bs\gamma}(j)
		+ |\delta_{\bs\gamma}(i) - \delta_{\bs\gamma}(j)| \\
		&= - \bigl(\overline\delta_{\bs\gamma, i, j} + \underline\delta_{\bs\gamma, i, j}\bigr)
		+ \overline\delta_{\bs\gamma, i, j} - \underline\delta_{\bs\gamma, i, j} \\
		&= -2\underline\delta_{\bs\gamma, i, j} \\
		&\le -2\underline\delta_{\bs\gamma} \\
		&\le 2\bigl(\overline\delta_{\bs\gamma} - \underline\delta_{\bs\gamma}\bigr).
	\end{align*}
	
	If $\delta_{\bs\gamma}(i)\delta_{\bs\gamma}(j) \le 0$, then
	\begin{align*}
		d_{\bs\gamma}(i, j)
		&= 2|\delta_{\bs\gamma}(i) - \delta_{\bs\gamma}(j)| \\
		&\le 2\bigl(\overline\delta_{\bs\gamma} - \underline\delta_{\bs\gamma}\bigr).
	\end{align*}
	
	Taken together, we have
	$d_{\bs\gamma}(i, j) \le 2\bigl(\overline\delta_{\bs\gamma} - \underline\delta_{\bs\gamma}\bigr)$
	in general.
	On the other hand, if $\delta_{\bs\gamma}(i) = \overline\delta_{\bs\gamma}$
	and $\delta_{\bs\gamma}(j) = \underline\delta_{\bs\gamma}$,
	then
	\begin{align*}
		d_{\bs\gamma}(i, j)
		&= \bigl|\overline\delta_{\bs\gamma}\bigr|
		+ \bigl|\overline\delta_{\bs\gamma} - \underline\delta_{\bs\gamma}\bigr|
		+ \bigl|\underline\delta_{\bs\gamma}\bigr| \\
		&= \overline\delta_{\bs\gamma}
		+ \bigl(\overline\delta_{\bs\gamma} - \underline\delta_{\bs\gamma}\bigr)
		- \underline\delta_{\bs\gamma} \\
		&= 2\bigl(\overline\delta_{\bs\gamma} - \underline\delta_{\bs\gamma}\bigr).
	\end{align*}
This completes the proof.
\end{proof}
\begin{theorem}\label{thm:rho&delta}
	For all $\bs\gamma \in \Gamma_{n, n}$,
	\begin{equation*}
		\widehat\rho_2(\bs\gamma) = 1 - \frac{\overline\delta_{\bs\gamma} - \underline\delta_{\bs\gamma}}{n}.
	\end{equation*}
\end{theorem}
\begin{proof}
	This follows immediately from \Cref{thm:rho2=maxd,thm:maxd=diff}.
\end{proof}

The following arguments (from Definition 7.6 to Theorem 7.15) refer to \cite[Section I]{AnaCom}.

\begin{definition}
	A \emph{combinatorial class} is a set $\mathcal A$ on which a \emph{size function} $|\mathord{\,\cdot\,}|:\mathcal A\to\N$ is defined so that $\{\alpha \in \mathcal A : |\alpha|=k\}$ is finite for all $k\in\N$.
	Unless confusion arises,
	we simply say a \emph{class} instead of a combinatorial class.
	
	Any subset $\mathcal B\subset\mathcal A$ is also a class
	with its size function defined as in $\mathcal A$.
	The \emph{counting sequence} $\{a_k\}$ of $\mathcal A$ is defined by 
	\[ a_k = \#\{\alpha \in \mathcal A : |\alpha| = k\} \qquad (k \in \N), \]
	and the \textit{ordinary generating function (OGF)} $A(x) \in \Z[[x]]$ of $\mathcal A$ is by
	\[ A(x) = \sum_{k = 0}^\infty a_k x^k. \]
\end{definition}
\begin{definition}
	Let $\mathcal A$ and $\mathcal B$ be two classes.
	A map $\phi: \mathcal A \to \mathcal B$
	is called a \textit{homomorphism}
	between $\mathcal A$ and $\mathcal B$ if $|\alpha| = |\phi(\alpha)|$ for all $\alpha \in \mathcal A$.
	If, in addition, $\phi$ is bijective,
	then we call $\phi$ an \textit{isomorphism},
	say that $\mathcal A$ and $\mathcal B$ are \textit{isomorphic},
	or write $\mathcal A \cong \mathcal B$.
\end{definition}
\begin{remark}\label{isomorphic <=> same counting sequences}
	Let $\mathcal A$ and $\mathcal B$ be
	two classes,
	$\{a_k\}$ and $\{b_k\}$ their counting sequences,
	and $A(x)$ and $B(x)$ their OGFs, respectively.
	We can easily see that the following three statements are equivalent:
	\begin{enumerate}
		\item $\mathcal A \cong \mathcal B$.
		\item $a_k = b_k$ for all $k \in \N$.
		\item $A(x) = B(x)$.
	\end{enumerate}
\end{remark}
\begin{definition}
	The \textit{neutral class} $\mathcal E$
	and the \textit{atomic class} $\mathcal Z$
	are classes with $\#\mathcal E = \#\mathcal Z = 1$, $|\varepsilon| = 0$ ($\varepsilon \in \mathcal E$), and $|\zeta| = 1$ ($\zeta \in \mathcal Z$).
\end{definition}
\begin{remark}\label{thm:OGF1}
	The OGFs of $\mathcal E$ and $\mathcal Z$
	are $1$ and $x$ in $\Z[[x]]$, respectively.
\end{remark}
\begin{definition}\label{def:combinatorial sum}
	Let $\{A_i\}$ be a set of classes, where $i$ runs over some index set $I$.
	If $\mathcal B = \{(i, \alpha) : i \in I, \, \alpha \in \mathcal A_i \}$ is also a class with its size function defined by $|(i, \alpha)| = |\alpha|$, then we call $\mathcal B$ the \emph{combinatorial sum} (or simply the \emph{sum}) of $\{\mathcal A_i\}$ and write $\mathcal B = \bigsqcup_{i \in I} \mathcal A_i$.
	In particular, if $I = \{1,\ldots,k\}$,
	then $\mathcal B$ is always a class,
	and we may write $\mathcal B = \mathcal A_1 +\cdots+ \mathcal A_k$.
\end{definition}
\begin{definition}\label{def:cartesian product}
	The \emph{Cartesian product} (or simply the \emph{product}) $\mathcal A_1 \times\cdots\times \mathcal A_k$ of $k$ classes $\mathcal A_1,\ldots,\mathcal A_k$ is the class $\{(\alpha_1,\ldots,\alpha_k):\alpha_1\in\mathcal A_1,\ldots,\alpha_k\in\mathcal A_k\}$ whose size function is defined by $|(\alpha_1,\ldots, \alpha_k)| = |\alpha_1| +\cdots + |\alpha_k|$.
	
	For a class $\mathcal A$ and $k\in\N_+$, we may write $\mathcal A^k$ instead of $\mathcal A \times\cdots\times \mathcal A$ ($k$ times).
	Let $\mathcal A^0=\mathcal E = \{\varepsilon\}$.
	If a class $\mathcal B = \bigsqcup_{i \in \N} \mathcal A^i$ exists,
	then we call $\mathcal B$ a \textit{sequence class} of $\mathcal A$,
	and write $\mathcal B = \Seq(\mathcal A)$.
\end{definition}

\begin{remark}\label{thm:OGF of cartesian product}
	If $A_1(x),\ldots,A_k(x)$
	are the OGFs of classes $\mathcal A_1,\ldots,\mathcal A_k$, respectively,
	then the OGFs of $\mathcal A_1 +\ldots+ \mathcal A_k$ and $\mathcal A_1 \times\cdots\times \mathcal A_k$
	are $A_1(x) +\cdots+ A_k(x)$ and $A_1(x)\cdots A_k(x)$, respectively.
\end{remark}

\begin{theorem}\label{thm:sequence class}
{\normalfont (See \cite[Section I.2.1]{AnaCom} for reference.)}
	Let $\{a_i\}$ be the counting sequence of a class $\mathcal A$.
	Then $\Seq(\mathcal A)$ exists if and only if $a_0 = 0$.
\end{theorem}
\begin{theorem}\label{thm:OGF of sequence class}
{\normalfont (See \cite[Section I.2.2, Theorem I.1]{AnaCom} for the proof.)}
	Let $A(x)$ be the OGF of a class $\mathcal A$ and assume that $\Seq(\mathcal A)$ exists.
	Then the OGF of $\Seq(\mathcal A)$ is $1 / (1 - A(x))$.
\end{theorem}

\begin{definition}
	Define a class $\mathcal G$ by
	\begin{align*}
		\mathcal G &= \bigcup_{i = 0}^\infty \Gamma_{i, i},\\
		|\bs\gamma| &= i \qquad (\bs\gamma \in \Gamma_{i, i}).
	\end{align*}
	For $k, l \in \N$,
	let $\mathcal G_{k, l} = \{\bs\gamma \in \mathcal G : -k\le\underline\delta_{\bs\gamma},\, \overline\delta_{\bs\gamma}\le l\}$
	and $G_{k, l}(x)$ be the OGF of $\mathcal G_{k, l}$.
	For $\bs\gamma = (\gamma_1, \ldots, \gamma_{2i}) \in \Gamma_{i, i}$ ($i \ge 1$), define
	\begin{align*}
		\lambda^+(\bs\gamma) &= (0, \gamma_1, \ldots, \gamma_{2i}, 1)\in\Gamma_{i+1,i+1}, \\
		\lambda^-(\bs\gamma) &= (1, \gamma_1, \ldots, \gamma_{2i}, 0)\in\Gamma_{i+1,i+1}.
	\end{align*}
	Put $\lambda^+(e) = (0, 1) \in \Gamma_{1, 1}$ and $\lambda^-(e) = (1, 0) \in \Gamma_{1, 1}$ for $e \in \Gamma_{0,0}$.
	Note that $\lambda^+$ and $\lambda^-$ are injective on $\mathcal G$.
\end{definition}
\begin{lemma}\label{thm:pmcong}
	For any $\mathcal H \subset \mathcal G$, $\mathcal H \times \mathcal Z \cong \lambda^+(\mathcal H) \cong \lambda^-(\mathcal H)$.
\end{lemma}
\begin{proof}
	Since $|(\bs\gamma, \zeta)| = |\bs\gamma| + |\zeta| = |\bs\gamma| + 1 = |\lambda^+(\bs\gamma)|$
	for all $(\bs\gamma, \zeta) \in \mathcal H \times \mathcal Z$,
	the bijection $\nu^+: \mathcal H \times \mathcal Z \to \lambda^+(\mathcal H)$
	defined by $\nu^+(\bs\gamma, \zeta) = \lambda^+(\bs\gamma)$
	is a homomorphism,
	hence $\mathcal H \times \mathcal Z \cong \lambda^+(\mathcal H)$.
	Similarly, the bijection $\nu^-: \mathcal H \times \mathcal Z \to \lambda^-(\mathcal H)$
	defined by $\nu^-(\bs\gamma, \zeta) = \lambda^-(\bs\gamma)$
	is a homomorphism,
	hence $\mathcal H \times \mathcal Z \cong \lambda^-(\mathcal H)$.
\end{proof}
\begin{corollary}\label{thm:pmOGF}
	If $H(x)$ is the OGF of $\mathcal H \subset \mathcal G$,
	then the OGFs of $\lambda^+(\mathcal H)$ and $\lambda^-(\mathcal H)$ are both equal to $xH(x)$.
\end{corollary}
\begin{proof}
	By \Cref{isomorphic <=> same counting sequences,thm:pmcong},
	the OGFs of $\lambda^+(\mathcal H)$ and $\lambda^-(\mathcal H)$ are equal to that of $\mathcal H \times \mathcal Z$, which equals $xH(x)$ by \Cref{thm:OGF1,thm:OGF of cartesian product}.
\end{proof}
\begin{lemma}\label{thm:Gcong}
	For all $k, l \in \N$, we have
	$\mathcal G_{k + 1, 0} \cong \Seq(\lambda^-(\mathcal G_{k, 0}))$,
	$\mathcal G_{0, l + 1} \cong \Seq(\lambda^+(\mathcal G_{0, l}))$,
	and $\mathcal G_{k + 1, l + 1} \cong \Seq(\lambda^-(\mathcal G_{k, 0}) + \lambda^+(\mathcal G_{0, l}))$.
\end{lemma}
\begin{proof}
Define a map $\sigma:\mathcal G_{k + 1, 0}\to \Seq(\lambda^-(\mathcal G_{k, 0}))$ by $\sigma(e)=(0,\varepsilon)$
	where $\varepsilon \in \mathcal E = (\lambda^-(\mathcal G_{k, 0}))^0 $,
	and by
\begin{equation*}
\sigma(\bs\gamma) = \left(p, (\bs\gamma_{j_0:j_1}, \ldots, \bs\gamma_{j_{p - 1}:j_p})\right)
\qquad (\bs\gamma\in\Gamma_{i,i};\ i\ge 1)
\end{equation*}
where $\{j_0,\ldots,j_p\}=\{j\in\{0,\ldots,2i\}:\delta_{\bs\gamma}(j)=0\}$ and $0=j_0<\cdots<j_p=2i$.
It follows from definition that $\sigma$ is bijective and $|\sigma(\bs\gamma)|=|\bs\gamma|$ for all $\bs\gamma\in\mathcal G_{k + 1, 0}$, so that $\sigma$ is an isomorphism, i.e., $\mathcal G_{k + 1, 0} \cong \Seq(\lambda^-(\mathcal G_{k, 0}))$.
We can similarly show that $\mathcal G_{0, l + 1} \cong \Seq(\lambda^+(\mathcal G_{0, l}))$ and $\mathcal G_{k + 1, l + 1} \cong \Seq(\lambda^-(\mathcal G_{k, 0}) + \lambda^+(\mathcal G_{0, l}))$.
\end{proof}

\begin{lemma}\label{thm:Qrel}
	For all $k, l \in \N_+$,
	\begin{equation}\label{eq:Qrel}
		Q_{k + 1}(x) Q_{l + 1}(x)
		- x^2 Q_{k - 1}(x) Q_{l - 1}(x)
		= Q_{k + l + 1}(x).
	\end{equation}
\end{lemma}
\begin{proof}
	Since $Q_2(x) = Q_1(x) - xQ_0(x) = 1 - x$ by \Cref{def:fibonacci}, we have 
	\begin{align}
		Q_{k + 1}(x) Q_2(x) \notag
		&= Q_{k + 1}(x) - xQ_{k + 1}(x) \notag \\
		&= Q_{k + 1}(x) - x(Q_k(x) - xQ_{k - 1}(x)) \notag \\
		&= Q_{k + 1}(x) - xQ_k(x) + x^2 Q_{k - 1}(x) \notag \\
		&= Q_{k + 2}(x) + x^2 Q_{k - 1}(x) \label{eq:Qrel:l=1} \\
		&= Q_{k + 2}(x) + x^2 Q_{k - 1}(x) Q_0(x), \notag
	\end{align}
	hence \Cref{eq:Qrel} holds for $l = 1$. We also have
	\begin{align*}
		Q_{k + 1}(x) Q_3(x)
		&= Q_{k + 1}(x) (Q_2(x) - x Q_1(x)) \\
		&= Q_{k + 1}(x) Q_2(x) - x Q_{k + 1}(x) \\
		&= Q_{k + 2}(x) + x^2 Q_{k - 1}(x) - x Q_{k + 1}(x) \\
		&= Q_{k + 3}(x) + x^2 Q_{k - 1}(x) \\
		&= Q_{k + 3}(x) + x^2 Q_{k - 1}(x) Q_1(x)
	\end{align*}
	by \Cref{eq:Qrel:l=1}, hence \Cref{eq:Qrel} holds for $l = 2$.
	
	By \Cref{def:fibonacci},
	we have $\bs{Q}_{j + 1}(x) = R \bs{Q}_j(x)$ for all $j \in \N$, where
	\begin{equation*}
		\bs{Q}_i(x) = \mqty(Q_i(x) \\ Q_{i + 1}(x)),
		\quad R = \mqty(0 & 1 \\ -x & 1).
	\end{equation*}	
	If \Cref{eq:Qrel} holds for $l = i\in\N_+$ and $l=i + 1$, or equivalently 
	\[ Q_{k + 1}(x) \bs{Q}_{i + 1}(x) - x^2 Q_{k - 1}(x) \bs{Q}_{i - 1}(x) = \bs{Q}_{k + i + 1}(x) \]
	holds, then
	\begin{align*}
		& Q_{k + 1}(x) \bs{Q}_{i + 2}(x) - x^2 Q_{k - 1}(x) \bs{Q}_i(x) \\
		&= Q_{k + 1}(x) R \bs{Q}_{i + 1}(x) - x^2 Q_{k - 1}(x) R \bs{Q}_{i - 1}(x) \\
		&= R(Q_{k + 1}(x) \bs{Q}_{i + 1}(x) - x^2 Q_{k - 1}(x) \bs{Q}_{i - 1}(x)) \\
		&= R \bs{Q}_{k + i + 1}(x) \\
		&= \bs{Q}_{k + i + 2}(x),
	\end{align*}
	hence \Cref{eq:Qrel} holds for $l = i + 2$.
	The claim follows by induction.
\end{proof}
\begin{lemma}\label{thm:Qsum}
	For all $k \in \N$,
	\begin{equation*}
		\sum_{0 \le i < k} Q_i(x) Q_{k - i - 1}(x) = - Q_{k + 1}'(x).
	\end{equation*}
\end{lemma}
\begin{proof}
	Define $\mathfrak Q(x, t)\in(\mathbb Z[x])[[t]]$ as
	\[ \mathfrak Q(x, t) = \sum_{k = 0}^\infty Q_k(x) t^k. \]
	Since
	\begin{align*}
		\mathfrak Q(x, t) &= Q_0(x) + Q_1(x) t + \sum_{k = 2}^\infty Q_k(x) t^k, \\
		t\mathfrak Q(x, t) &= Q_0(x)t + \sum_{k = 2}^\infty Q_{k - 1}(x) t^k, \\
		xt^2\mathfrak Q(x, t) &= x\sum_{k = 2}^\infty Q_{k - 2}(x) t^k,
	\end{align*}
	we have
	\begin{align*}
		(1 - t + xt^2) \mathfrak Q(x, t)
		&= Q_0(x) + Q_1(x) t - Q_0(x) t \\
		& \quad + \sum_{k = 2}^\infty (Q_k(x) - Q_{k - 1}(x) + x Q_{k - 2}(x)) t^k \\
		&= 1
	\end{align*}
	by \Cref{def:fibonacci}, hence
	\begin{equation}
		\mathfrak Q(x, t) = \frac1{1 - t + xt^2}.
	\end{equation}
	Therefore 
	\begin{align*}
		\sum_{k = 0}^\infty Q_k'(x) t^k
		&= \pdv{x} \mathfrak Q(x, t)
		= -\frac{t^2}{(1 - t + xt^2)^2} \\
		&= -t^2 (\mathfrak Q(x, t))^2
		= -\sum_{k = 0}^\infty \Biggl(\sum_{0 \le i < k} Q_i(x) Q_{k - i - 1}(x)\Biggr) t^{k + 1},
	\end{align*}
	which implies the claim.
\end{proof}
\begin{proposition}\label{thm:OGFofG0}
	For all $k \in \N$,
	\begin{equation}\label{eq:OGFofG0}
		G_{k, 0}(x) = G_{0, k}(x) = \frac{Q_k(x)}{Q_{k + 1}(x)}.
	\end{equation}
\end{proposition}
\begin{proof}
	Since $G_{0, 0}(x) = 1 = Q_0(x) / Q_1(x)$, \Cref{eq:OGFofG0} holds for $k=0$.
	
	Suppose \Cref{eq:OGFofG0} holds for some $k \in \N$.
	Since $\mathcal G_{k + 1, 0} \cong \Seq(\lambda^-(\mathcal G_{k, 0}))$ by \Cref{thm:Gcong},
	we have
	\[ G_{k + 1, 0}(x) = \frac{1}{1 - x G_{k, 0}(x)} 
	= \frac{ Q_{k + 1}(x) }{ Q_{k + 1}(x) - x Q_k(x) }
	= \frac{ Q_{k + 1}(x) }{ Q_{k + 2}(x) }
	\]
	by \Cref{def:fibonacci,thm:OGF of sequence class,thm:pmOGF}.
	Similarly,
	since $\mathcal G_{0, k + 1} \cong \Seq(\lambda^+(\mathcal G_{0, k}))$ by \Cref{thm:Gcong}, we obtain
	\begin{equation*}
		G_{0, k + 1}(x)
		= \frac{ Q_{k + 1}(x) }{ Q_{k + 2}(x) }.
	\end{equation*}
	Therefore, \Cref{eq:OGFofG0} holds for $k+1$ in place of $k$, and the proof is complete.
\end{proof}
\begin{proposition}\label{thm:OGFofG}
	For all $k, l \in \N$,
	\begin{equation}\label{eq:OGFofG}
		G_{k, l}(x) = \frac{Q_k(x)Q_l(x)}{Q_{k + l + 1}(x)}.
	\end{equation}
\end{proposition}
\begin{proof}
	If $k = 0$ or $l = 0$, then \Cref{eq:OGFofG} holds by \Cref{thm:OGFofG0}.
	
	Suppose $k \ge 1$ and $l \ge 1$.
	Since $\mathcal G_{k, l} \cong \Seq(\lambda^-(\mathcal G_{k - 1, 0}) + \lambda^+(\mathcal G_{0, l - 1}))$ by \Cref{thm:Gcong},
	we have
	\[ G_{k, l}(x) = \frac{1}{1 - x G_{k - 1, 0}(x) - x G_{0, l - 1}(x)} \]
	by \Cref{thm:OGF of cartesian product,thm:OGF of sequence class,thm:pmOGF}, where
	\[
		G_{k - 1, 0}(x) = \frac{Q_{k - 1}(x)}{Q_k(x) },\qquad
		G_{0, l - 1}(x) = \frac{Q_{l - 1}(x)}{Q_l(x)}
	\]
	by \Cref{thm:OGFofG0}.
	Hence
	\begin{align*}
		G_{k, l}(x)
		&= \frac{Q_k(x) Q_l(x)}{Q_k(x) Q_l(x) - x Q_{k - 1}(x) Q_l(x) - x Q_k(x) Q_{l - 1}(x)} \\
		&= \frac{Q_k(x) Q_l(x)}{(Q_k(x) - xQ_{k - 1}(x))(Q_l(x) - xQ_{l - 1}(x)) - x^2Q_{k - 1}(x)Q_{l - 1}(x)} \\
		&= \frac{Q_k(x) Q_l(x)}{Q_{k + 1}(x)Q_{l + 1}(x) - x^2Q_{k - 1}(x)Q_{l - 1}(x)} \\
		&= \frac{Q_k(x) Q_l(x)}{Q_{k + l + 1}(x)}
	\end{align*}
	by \Cref{def:fibonacci,thm:Qrel}.
\end{proof}
\begin{proposition}\label{thm:OGFwtG}
	For $k \in \N$, the OGF of $\wt{\mathcal G}_k = \{ \bs\gamma \in \mathcal G : \overline\delta_{\bs\gamma} - \underline\delta_{\bs\gamma} \le k \}$ is
	\[ \wt G_k(x) = \frac{Q_{k + 1}'(x)}{Q_k(x)} - \frac{Q_{k + 2}'(x)}{Q_{k + 1}(x)}. \]
\end{proposition}
\begin{proof}
	For $i, j \in \N$,
	let $\wt{\mathcal G}_{i, j} = \{ \bs\gamma \in \mathcal G : -i = \underline\delta_{\bs\gamma} ,\, \overline\delta_{\bs\gamma} \le j \}$
	and $\wt G_{i, j}$ be the OGF of $\wt{\mathcal G}_{i, j}$.
	Since $\wt{\mathcal G}_{0, j} = \mathcal G_{0, j}$, $ \mathcal G_{i, j} \cong \mathcal G_{i - 1, j} + \wt{\mathcal G}_{i, j}$ if $i\ge 1$, and 
	$
		\wt{\mathcal G}_k = \bigsqcup_{i = 0}^k \wt{\mathcal G}_{i, k - i}
	$
	by definition, we have
	\begin{align*}
		\wt G_k(x)
		&= \sum_{i = 0}^k \wt G_{i, k - i}(x) \\
		&= G_{0, k}(x) + \sum_{1 \le i < k + 1} (G_{i, k - i}(x) - G_{i - 1, k - i}(x)) \\
		&= \sum_{0 \le i < k + 1} G_{i, k - i}(x) - \sum_{0 \le i < k} G_{i, k - i - 1}(x)\\
		&= \sum_{0 \le i < k + 1} \frac{Q_i(x)Q_{k - i}(x)}{Q_{k + 1}(x)} - \sum_{0 \le i < k} \frac{Q_i(x)Q_{k - i - 1}(x)}{Q_k(x)}\\
		&= \frac{Q_{k + 1}'(x)}{Q_k(x)} - \frac{Q_{k + 2}'(x)}{Q_{k + 1}(x)}
	\end{align*}
	by \Cref{thm:OGF of cartesian product,thm:Qsum,thm:OGFofG}.
\end{proof}
\begin{theorem}\label{thm:original_conjecture}
	For $k = 0,\ldots,n$, we have
	\begin{equation*}
		\#\biggl\{ \bs\gamma \in \Gamma_{n, n} : \widehat\rho_2(\bs\gamma) \ge 1 - \frac{k}{n} \biggr\}
		= [x^n]\biggl(\frac{Q_{k + 1}'(x)}{Q_{k}(x)} - \frac{Q_{k + 2}'(x)}{Q_{k + 1}(x)}\biggr).
	\end{equation*}
\end{theorem}
\begin{proof}
	We have
	\begin{align*}
		\#\biggl\{ \bs\gamma \in \Gamma_{n, n} : \widehat\rho_2(\bs\gamma) \ge 1 - \frac{k}{n} \biggr\}
		&= \#\{ \bs\gamma \in \Gamma_{n, n} : \overline\delta_{\bs\gamma} - \underline\delta_{\bs\gamma} \le k \} \\
		&= \#\{ \bs\gamma \in \wt{\mathcal G}_k : |\bs\gamma| = n \} \\
		&= [x^n] \wt G_k(x) \\
		&= [x^n]\biggl(\frac{Q_{k + 1}'(x)}{Q_{k}(x)} - \frac{Q_{k + 2}'(x)}{Q_{k + 1}(x)}\biggr)
	\end{align*}
	by \Cref{thm:rho&delta,thm:OGFwtG}.
\end{proof}

Note that \Cref{thm:original_conjecture} is exactly \Cref{thm:rho_hat.fibonacci}.

\section*{Acknowledgement}
This study was partially supported by JSPS KAKENHI Grant Numbers JP21K15762, JP15K04814, and JP20K03509.

%
\vskip24pt
\small
%
\centerline{\bf References}
\vskip12pt
\begin{enumerate}
\renewcommand{\labelenumi}{[\arabic{enumi}]}
\renewcommand{\makelabel}{\rm}
\setcounter{enumi}{0}
\setlength{\itemsep}{-3pt}
\setlength{\parsep}{0cm}

\bibitem{anderson62}
Anderson, T. W.,
On the distribution of the two-sample Cram\'{e}r-von Mises criterion,
Ann. Math. Stat. {\bf 33} (1962), 1148--1159.
\bibitem{berger14}
Berger, V. W. and Y. Zhou,
Kolmogorov-Smirnov Test: Overview,
Wiley StatsRef: Statistics Reference Online
(2014).
\bibitem{AnaCom}
Flajolet, P. and R. Sedgewick,
Analytic Combinatorics,
Cambridge University Press (2009).
\bibitem{Hand}
Hand, D. J.,
Assessing the performance of classification methods,
Int. Stat. Rev.
{\bf 80} (2012), 400--414.
\bibitem{hsu47}
Hsu, P. L. and H. Robbins,
Complete convergence and the law of large numbers,
Proc. Nat. Acad. Sci. U.S.A.
{\bf 33} (1947), 25--31.
\bibitem{johno21}
Johno, H. and K. Nakamoto,
Decision tree-based estimation of the overlap of two probability distributions, J. Math. Sci. Univ. Tokyo {\bf 30} (2023), 21--54.
\bibitem{mann47}
Mann, H. B. and D. R. Whitney,
On a test of whether one of two random variables is stochastically larger than the other,
Ann. Math. Stat. {\bf 18} (1947), 50--60.
\bibitem{Gunar}
Schr\"{o}er, G. and D. Trenkler,
Exact and randomization distributions of \\
Kolmogorov-Smirnov tests two or three samples,
Comput. Stat. Data Anal. {\bf 20} (1995), 185--202.
\bibitem{serfling80}
Serfling, R. J.,
Approximation Theorems of Mathematical Statistics,
John Wiley \& Sons, Inc. (1980).
\bibitem{A115139}
Sloane, N. J. A. and The OEIS Foundation Inc.,
entry A115139 in the On-Line Encyclopedia of Integer Sequences,
\url{https://oeis.org/A115139}.
\bibitem{snedecor89}
Snedecor, G. W. and W. G. Cochran,
Statistical Methods,
Iowa State University Press (1989).
\bibitem{welch47}
Welch, B. L.,
The generalization of `Student's' problem when several different population variances are involved,
Biometrika {\bf 34} (1947) 28--35.

\end{enumerate}

\parindent5cm
Atsushi Komaba
and
Hisashi Johno

Department of Radiology

Faculty of Medicine

University of Yamanashi

Shimokato 1110, Chuo, Yamanashi 409-3898, Japan

Email: fiveseven.lambda@gmail.com

\phantom{Email: }johnoh@yamanashi.ac.jp
\vskip12pt
Kazunori Nakamoto

Center for Medical Education and Sciences

Faculty of Medicine

University of Yamanashi

Shimokato 1110, Chuo, Yamanashi 409-3898, Japan

Email: nakamoto@yamanashi.ac.jp
\end{document}